\documentclass[11pt,reqno]{amsart}
\usepackage[usenames]{color}
\usepackage{amsmath,pdfsync,verbatim,graphicx,epstopdf,enumerate,fancybox}
\usepackage[colorlinks=true]{hyperref}
\hypersetup{urlcolor=blue, citecolor=red, linkcolor=red}
\usepackage[notref,notcite]{}
\usepackage[framemethod=tikz]{mdframed}
\newtheorem{theorem}{Theorem}[section]
\newtheorem{lemma}[theorem]{Lemma}

\newtheorem{definition}[theorem]{Definition}
\newtheorem{remark}[theorem]{Remark}

\numberwithin{equation}{section}

\newcommand{\I}{\mathrm{i}}
\newcommand{\vf}{\varphi}

\newcommand{\PD}{\partial}

\newcommand{\Fc}{\mathcal{F}}

\newcommand{\Ic}{\mathcal{I}}
\newcommand{\Jc}{\mathcal{J}}

\newcommand{\Sc}{\mathcal{S}}

\newcommand{\Rb}{\mathbb{R}}
\newcommand{\Sb}{\mathbb{S}}

\newcommand{\Beq}{\begin{equation}}
\newcommand{\Eeq}{\end{equation}}
\newcommand{\beq}{\begin{equation*}}
\newcommand{\eeq}{\end{equation*}}
\newcommand{\bal}{\begin{align}}
\newcommand{\eal}{\end{align}}

\renewcommand{\l}{\langle}
\renewcommand{\r}{\rangle}

\newcommand{\bp}{\begin{prob}}
	\newcommand{\ep}{\end{prob}}
\newcommand{\bpr}{\begin{proof}}
	\newcommand{\epr}{\end{proof}}

\usepackage{amsmath,tikz}
\usetikzlibrary{matrix}

\newcommand*\Rn{\mathbb{R}^n}

\newcommand{\sch}{\mathcal{S}}
\newcommand{\tn}{T\mathbb{S}^{n-1}}
\renewcommand{\d}{\mathrm{d}}	

\usepackage[centerlast,small,sc]{caption}
\title[Integral moment transforms in $\Rb^n$]{Injectivity and range description of first $(k+1)$ integral moment transforms over $m$-tensor fields in $\Rb^n$}

\author[Rohit Kumar Mishra and Suman Kumar Sahoo]{Rohit Kumar Mishra$^\ast$ and Suman Kumar Sahoo$^\dagger$}
	
	\email{rohit.mishra@uta.edu, rohittifr2011@gmail.com, suman@tifrbng.res.in}
	\address{$^\ast$ University of Texas at Arlington, Arlington, TX, United States
		\newline\indent$^\dagger$TIFR Centre for Applicable Mathematics, Sharada Nagar, Chikkabommasandra,\newline\indent\hspace{0mm} Yelahanka New Town, Bangalore, India}

\begin{document}
	\maketitle
\begin{abstract}
	In this work, we prove a new decomposition result for rank $m$ symmetric tensor fields which generalizes the well known solenoidal and potential decomposition of tensor fields. This decomposition is then used to describe the kernel and to prove an injectivity result for first $(k+1)$  integral moment transforms of symmetric $m$-tensor fields in $\Rb^n$. Additionally, we also present  a range characterization for first $(k+1)$ integral moment transforms in terms of the John's equation.
\end{abstract}
%
	\section{Introduction}
The space of covariant symmetric $m$-tensor fields on $\Rb^n$ with components in Schwartz space $\Sc(\Rb^n)$ will be denoted by  $\Sc(S^m)$. In standard Euclidean coordinates, any element $f \in \mathcal{S}(S^m)$ can be written as
$$ f(x) = f_{i_1\dots i_m}(x) dx^{i_1} \cdots dx^{i_m}$$
with $f_{i_1 \dots i_m} \in \Sc(\Rb^n)$ are symmetric in its components. For repeated indices, Einstein summation convention will be assumed throughout this article. Also, we will not distinguish between covariant and contravariant tensors as we are working with the Euclidean metric. 

The space of oriented lines in $\Rb^n$ is parametrized by points of the tangent bundle of unit sphere $\Sb^{n-1}$ and it is denoted by
$$T{\Sb}^{n-1}=\{(x,\xi)\in{\Rb}^n\times{\Rb}^n\mid |\xi|=1,\langle x,\xi\rangle=0\}.
$$
For each $(x,\xi)\in T{\Sb}^{n-1}$, we have a unique line $\{x+t\xi\mid t\in{\Rb}\}$ passing through point $x$ and in the direction $\xi$.

For a non-negative integer $q \geq 0$, the $q$-th integral moment transform of a symmetric $m$-tensor field is the function $I^q :{\Sc}(S^m)\rightarrow{\Sc}(\tn)$ given by  \cite{Sharafutdinov_Generalized_Tensor_Fields}:
\begin{equation}\label{eq:definition of momentum ray transform}
(I^q f)(x,\xi)=\int\limits_{-\infty}^\infty t^q\langle f(x+t\xi),\xi^m\rangle dt = \int\limits_{-\infty}^\infty t^q f_{i_1\dots i_m}(x+t\xi)\,\xi^{i_1} \cdots \xi^{i_m} dt.
\end{equation}
In the above equation, $\langle f, \xi^m \rangle$ actually means $\langle f, \xi^{\otimes m} \rangle$,  where $\xi^{\otimes m}$ denotes $m$-times tensor product of $\xi$ with itself. 
\vspace{2mm}

\noindent The collection of first $(k+1)$ integral moment transforms of $f \in \Sc(S^m)$ is  denoted by $\Ic^k f$, more specifically, the operator $\Ic^k: \mathcal{S}(S^m)\rightarrow \left(\mathcal{S}(\tn)\right)^{k+1}$ defined by
\begin{align}\label{eq:definition of Ik moment transforms}
 \Ic^k(f)(x, \xi) =  (I^0 f(x, \xi),I^1 f(x, \xi),\dots, I^k f(x, \xi)), \quad \mbox{ for } (x, \xi) \in \tn.   
\end{align} 
The zeroth integral moment transform $I^0$ or $\Ic^0$ coincides with the well known longitudinal ray transform (also known as ray transform) of  symmetric $m$-tensor fields in $\Rb^n$. The problem of inverting the longitudinal ray transform (LRT) is primarily motivated from their appearance in several imaging problems, notably in medical imaging, seismic imaging, ocean imaging and many more. It is well known \cite{Sharafutdinov1994} that the LRT has  a non-trivial kernel (containing all potential tensor fields with certain decay at infinity) which tells that one cannot recover the  entire tensor field just from LRT data. On the other hand, the solenoidal part $f^s$ of a symmetric $m$-tensor field $f$ can be determined uniquely from the knowledge of $\Ic^0 f$. In this regard, explicit reconstruction algorithms have been studied by many researchers in various settings, please see \cite{Denisjuk1994,Denisjuk_Paper,Helgason_Book,Katsevich2006,Katsevich2007, Katsevich2013,Monard_Mishra_2019,Monard1,Monard2,Monard,Monard_2019,Palamodov2009,Schuster2000,Sharafutdinov2007,Svetov2012,Tuy1983,Vertgeim2000} and references therein. In addition to these explicit schemes,  approximate inversion methods (such as microlocal inversion) have also been developed extensively to recover the solenoidal part a symmetric $m$-tensor field, see \cite{Boman-Quinto-Duke,Boman1993,Greenleaf-Uhlmann-Duke1989,GU1990c,Krishnan2009a,Venky_and_Rohit,Krishnan-Quinto,Krishnan2009,Lan2003,Ramaseshan2004}.

It is evident from the non-injectivity of LRT that one needs more information (in addition to LRT) for the full recovery of a tensor field. In 1984, Sharafutdinov \cite{Sharafutdinov_Generalized_Tensor_Fields} introduced integral moment transforms (see \eqref{eq:definition of momentum ray transform}) and showed that the collection of first $(m+1)$ integral moment transforms, $\Ic^m$, is injective over symmetric $m$-tensor fields in $\Rb^n$. 
 For the scalar case $ (m=0) $, the integral moment transforms $ I^k\, (k>0) $  appear in the study of inversion of cone transforms and conical Radon transforms, see \cite{haltmeier2020conical,Kuchment_Fatma,Markus_Moon_conical_Radon_trans} and references there in. And the latter transforms arise in image reconstruction from the data obtained by  Compton cameras, which have potential applications in medical and industrial imaging. In \cite{Anuj_Rohit}, authors proved a support theorem and an injectivity result for first $(m+1)$ integral moment transforms of symmetric $m$-tensor fields on  simple real analytic Riemannian manifolds. Then in \cite{Francois_Rohit_Venky}, authors gave an inversion formula for integral moment transforms on a simple Riemannian surface. Later in the article \cite{Rohit_Kumar_Mishra}, author presented an explicit scheme for the recovery of a vector field in $\Rb^n$ using $n$-dimensional restricted data of first 2-integral moment transform of the unknown vector field. Most recently in a couple of papers \cite{Krishnan2018,Krishnan2019a}, authors studied first $(m+1)$ integral moment transforms and its properties  over $m$-tensor fields in a great detail. In \cite{Krishnan2018}, authors proved the invertibility together with stability estimates for the collection of first $(m+1)$ integral moment transform $\Ic^m$. In their second paper \cite{Krishnan2019a}, authors gave a detailed description of range for the operator $\Ic^m$.   

To the best of our knowledge, the study on the transform $\Ic^k$ over rank symmetric $m$-tensor fields is limited to cases $k=0$ and $k=m$ only. The current article addresses injectivity and range characterization questions for the intermediate cases $0 < k < m$ of the operator $\Ic^k$. 
It is well known that a symmetric $m$-tensor field $f$ can be decomposed uniquely into its potential part and solenoidal part. This decomposition is not closed in the sense that the solenoidal and the potential components of a tensor field $f$ are not in the Schwartz space even if $f$ is in the Schwartz space. Therefore, it is not possible to apply an iterative scheme (similar to \cite{Anuj_Rohit}) on the decomposition. 
To overcome this difficulty, we introduce $k$-potential tensor fields and $k$-solenoidal tensor fields  (see Definition \ref{def:k-potential and k-solenoidal}) by extending classical notions of potential and solenoidal tensor fields respectively. Then, we prove a decomposition result (see Theorem \ref{th:decomp_in_original_sp}) which shows that any symmetric $m$-tensor field $f$ can be decomposed uniquely into a $k$-potential tensor field and a $k$-solenoidal tensor field. 
 With the help of this decomposition theorem, we provide an explicit description of the kernel for the operator $\Ic^k$, see Theorem \ref{th:injectivity result}. Additionally, we also prove that the operator $\Ic^k$ is  injective over $(k+1)$-solenoidal tensor fields. 
 Our injectivity result generalizes the existing injectivity results for $\Ic^0$ (injective over solenoidal tensor fields) and $\Ic^m$ (injective over $m$-tensor fields).

 Apart from injectivity and invertibility issues, the range characterization questions are also very important in the field of integral geometry. 
For instance, the knowledge of range is essential in order to project measured data on the range before applying inversion algorithms. 
The second order differential operator (also known as John operator)
\begin{align}\label{def:John equation}
	J_{ij}=\frac{\PD^2}{\PD x^i\PD\xi^j}- \frac{\PD^2}{\PD x^j\PD\xi^i} \quad 1\le i,j\le n
\end{align}
shows up in the range characterization results for ray transform of functions  by Helgason  \cite{Helgason_Book} and of tensor fields  by Sharafutdinov \cite{Sharafutdinov1994} in $\mathbb{R}^n (n\ge 3)$. The John differential equation was first introduced by Fritz John \cite{John_Fritz_ultrahyperbolic} to study ultrahyperbolic differential equations in $\mathbb{R}^3$. 
The final goal of this article is to give a detailed description of the range for the operator $\Ic^k$ in terms of John's differential equations, see Theorem \ref{th:range characterisation for I-k}.

The rest of the article is organized as follows. In section \ref{sec:Def and notation}, we introduce some definitions and notation used throughout this work. Section \ref{sec:Decomposition Results} is devoted to the proof of decomposition theorem of symmetric $m$-tensor fields. The injectivity results and kernel description is discussed in section \ref{sec: Injectivity results}. Finally, section \ref{sec:range characterization} contains the proof of range characterization for the integral moment transform $\Ic^k$. \\

\noindent \textbf{Acknowledgements.} The second author would like to thank Jenn-Nan Wang for suggesting this problem during his visit to Taiwan in 2018 and he would also like to express his sincere gratitude towards  Vladimir A. Sharafutdinov for introducing the subject of this paper. Both authors would like to thank  Venky P. Krishnan for several fruitful discussions on the results of this article which helped us to improve the manuscript.\\

\noindent \textbf{Funding:} Both authors benefited from the  Venky P. Krishnan's SERB Matrics grant MTR/2017/000837.

\section{Definitions and notation}\label{sec:Def and notation}
In this section we introduce some important definitions and notation used throughout this article. Most of these definitions and notation can be found in the book ``Integral geometry of tensor fields" by Sharafutdinov \cite{Sharafutdinov1994} and also in the article \cite{Krishnan2018}.
\subsection{Some differential operators} Let $T^m(\Rb^n)$ denotes the space of $m$-tensors on $\Rb^n$. There is a natural projection of $T^m(\Rb^n)$ onto the space of symmetric tensors $S^m(\Rb^n)$, $\sigma : T^m(\Rb^n) \rightarrow S^m(\Rb^n)$ given by  
\begin{align}\label{eq:definition of sigma}
(\sigma v)_{i_1\dots i_m} = \frac{1}{m!}\sum_{\pi \in \Pi_m} v_{\pi(i_1)\dots \pi(i_m)} 
\end{align}
where $\Pi_m$ is the set of permutation of order $m$.\vspace{2mm}\\
For $x \in \Rb^n$,  we define the \textit{symmetric multiplication operators} $i_x :S^m(\Rb^n)  \rightarrow S^{m+1}(\Rb^n)$ by 
$$(i_x f)_{i_1i_2\dots, i_{m+1}} =\sigma(i_1, \dots, i_m, i_{m+1})(x_{i_{m+1}}f_{i_{1}i_{2}\dots i_{m}}).$$
In the same spirit, we also define the dual of $i_x$, \textit{the convolution operator}, $j_x :S^m(\Rb^n)  \rightarrow S^{m-1}(\Rb^n)$ by 
$$(j_x f)_{i_1i_2\dots i_{m-1}} =f_{i_1i_2\dots i_m}x^{i_m}.$$ 
The composition of these operators will be essential in the next section to prove the decomposition theorem and hence for the convenience of reader, we introduce the operators $i_{x^{\otimes k}} :  S^m(\Rb^n) \rightarrow S^{m+k}(\Rb^n)$ and $j_{x^{\otimes k}} :  S^{m+k}(\Rb^n) \rightarrow S^{m}(\Rb^n)$, for any fixed integer $k \geq 1$, as follows:
\begin{align*}
\left(i_{x^{\otimes k}}f\right)_{i_1i_2\dots i_{m+k}} &=\sigma(i_1, \dots, i_m \dots i_{m+k})(x_{i_{m+1}\dots x_{i_{m+k}}}f_{i_{1}i_{2}\dots i_{m}})\\
(j_{x^{\otimes k}}f)_{i_1i_2\dots i_{m}} &=x^{i_{m+1}}\dots x^{i_{m+k}}f_{i_{1}i_{2}\dots i_{m}i_{m+1}\dots i_{m+k} }.
\end{align*}
\noindent Next, we define two important first order differential operators on $C^\infty(S^m)$,  the space symmetric $m$-tensor fields whose components are $C^\infty$ smooth.  The operator of \textit{inner differentiation} or \textit{symmetrized derivative} $\d:C^\infty(S^m)\rightarrow C^\infty(S^{m+1})$ is defined by
	$$(\d u)_{i_1\dots i_mi_{m+1}} = \sigma(i_1, \dots , i_m) \left(\frac{\partial u_{i_1\dots i_m} }{\partial x_{i_{m+1}}}\right)$$
	where $\sigma$ is defined in equation \eqref{eq:definition of sigma}. \vspace{2mm}\\
	\noindent The \textit{divergence} operator $\delta:C^\infty(S^{m})\rightarrow C^\infty(S^{m-1})$ is 
	defined by the formula $$ (\delta u)_{i_1\dots i_{m-1}} = \sum_{j=1}^n \frac{\partial u_{i_1\dots i_{m-1}j} }{\partial x_{j}}.$$

 \subsection{Some properties of moment ray transforms}
 Note that the definition of $q$-th integral moment transform $I^q$ will make sense if we define them to $\Rb^n \times \Rb^n \setminus \{0\}$. For later use, we define the operator $J^q: \Sc(S^m) \longrightarrow C^\infty(\Rb^n \times (\Rb^{n} \setminus \{0\}))$ by extending $I^q$ to  $ \Rn \times \Rn\setminus\{0\} $ 
 \begin{equation}\label{eq:definition of Jk}
J^q f(x,\xi)=\int\limits_{-\infty}^\infty t^q\langle f(x+t\xi),\xi^m\rangle \, d t \quad\mbox{for}\quad(x,\xi)\in \mathbb{R}^{n} \times \mathbb{R}^{n}\setminus \{0\}.
\end{equation} 
It has been shown in \cite{Krishnan2018} that the data $(I^0 f, I^1 f, \dots , I^k f)$ and $(J^0 f, J^1 f, \dots , J^k f)$ are equivalent for any $0 \leq k \leq m$ and there is a explicit relation between these operators	\begin{equation}\label{eq:relation between Ik and Jk}
	(J^q\!f)(x,\xi)=|\xi|^{m-2q-1}\sum\limits_{\ell=0}^q(-1)^{q-\ell}{q\choose\ell}|\xi|^\ell \langle\xi,x\rangle^{q-\ell}\,(I^\ell\!f)
	\left(x-\frac{\langle x, \xi \rangle}{|\xi|^2}\xi,\frac{\xi}{|\xi|}\right).
	\end{equation}
In certain instances, it will be more convenient to work with the operator $J^q$ instead of $I^q$. One clearly evident advantage of working with functions $ J^k f $ is that the partial derivatives $ \frac{\PD}{\PD\xi^{i}} $ and $\frac{\PD}{\PD x^{i}}$ are well defined on $ J^k f $ for $ k=0,1,\dots,m$. \vspace{2mm} 	

\noindent The Fourier transform of a symmetric $m$-tensor field $f \in \Sc(S^m)$ is defined component-wise, that is, 
\begin{align*}
 \widehat{f}_{i_1\dots i_m}(y)  = \widehat{f_{i_1\dots i_m}}(y), \quad y \in \Rb^n  
\end{align*}
where $\widehat{h}(y)$ denotes the usual Fourier transform of a scalar function $h$ defined on $\Rb^n$. \vspace{2mm} 

\noindent The Fourier transform  $\Fc :\Sc(\tn) \longrightarrow \Sc(\tn)$ is defined as follows, see \cite[Section 2.1]{Sharafutdinov1994}:
\begin{align}\label{eq:Fourier transform on sphere bundle}
    \Fc (\vf) (y, \xi) =    \widehat{\vf}(y, \xi) = \frac{1}{(2 \pi)^{(n-1)/2}}\int_{\xi^\perp} e^{-i x\cdot y} \vf(x, \xi)\, dx
\end{align}
where $dx$ is the $(n-1)$-dimensional Lebesgue measure on the hyperplane $\xi^\perp = \{ x\in \Rb^n : \l x,\xi \r = 0\}.$

\noindent This definition of Fourier transform is used to compute the following Fourier transform of $q$-th integral moment transform of $f$:
\begin{align*}
    \widehat{I^q f}(y, \xi) = (2 \pi)^{1/2} i^q \langle \xi, \partial_y\rangle^q \langle \widehat{f}(y), \xi^m\rangle.
\end{align*}
For $q=0$, the above equality reduces to
\begin{align*}
    \widehat{I f}(y, \xi) = (2 \pi)^{1/2} \langle \widehat{f}(y), \xi^m\rangle.
\end{align*}
\section{Decomposition results}\label{sec:Decomposition Results}
We start this section  by defining two special tensor fields which are generalizations of the solenoidal tensor fields and potential tensor fields respectively. 
\begin{definition}[$k$-solenoidal and $k$-potential tensor fields]\label{def:k-potential and k-solenoidal}
For any fixed $1 \leq k \leq m$, a symmetric $m$-tensor field $f \in C^\infty(S^m)$ is said to be 
\begin{enumerate}
    \item $k$-solenoidal tensor field if 
    \begin{align*}
    \delta^k f = 0.
\end{align*}
\item $k$-potential tensor field if there exists a $(m-k)$-tensor field $v \in C^\infty(S^{m-k})$
such that    \begin{align*}
  f = \d^k v.
\end{align*}
\end{enumerate}
\end{definition}
\noindent For $k=1$, the $k$-solenoidal and $k$-potential tensor fields coincide with the usual solenoidal and potential tensor fields respectively. 

The goal of this section is to prove that any symmetric $m$-tensor field can be decomposed uniquely into its $k$-solenoidal part and $k$-potential part. This decomposition theorem extends the result \cite[Theorem 2.6.2]{Sharafutdinov1994} which gives a unique decomposition of a symmetric $m$-tensor fields into its solenoidal part and potential part. In fact both these results can be viewed as generalizations of the well-known Helmholtz (the name Helmholtz-Hodge decomposition also used widely) decomposition of vector field into divergence free part (solenoidal part) and curl free part (potential part). \vspace{2mm}\\
To prove the main decomposition result of this section, we need the following two lemmas:
\begin{lemma}\label{th: decomposition of f in frequency variable} 
	Let $f$ be a symmetric $m$-tensor field in $\Rb^n$ and $x \in \Rb^n$ be a non-zero vector. Then for $0 \leq k  \leq m$, there exist symmetric $m$-tensor field $g$ and symmetric $(m-k)$-tensor field $v$ such that the following decomposition of $f$ holds:
	\begin{align}\label{eq: decomposition of f in frequency variable}
	f  = g + i_{x^{\otimes k}} v 
	\end{align}
	where $g$ satisfies $j_{x^{\otimes k}} g =0$ and given by
	 	 	\begin{align}\label{eq: expression of g}
	 	&g_{i_1i_2\cdots i_m}\nonumber\\
	 	&= \sigma(i_1, \dots , i_m)\left( \delta_{i_1}^{j_1}\cdots \delta_{i_k}^{j_k} - \frac{x_{i_1}\dots x_{i_k}x^{j_1}\dots x^{j_k}}{|x|^{2k}}\right)\left(\delta^{j_{k+1}}_{i_{k+1}} - \frac{x^{j_{k+1}}x_{i_{k+1}} }{|x|^2}\right)\cdots \left(\delta^{j_{m}}_{i_{m}} - \frac{x^{j_{m}}x_{i_{m}} }{|x|^2}\right)f_{j_1 j_2 \dots j_m}.	
	 	\end{align}	 
	\end{lemma}
We skip the proof of this lemma as it can be achieved directly from the duality of the linear operators  $ i_{x^{\otimes  k}}: S^m(\Rn)\longrightarrow S^{m+k}(\Rn) $ and $ j_{x^{\otimes  k}}:S^m(\Rn)\longrightarrow S^{m-k}(\Rn)$, see also \cite[Lemma 2.6.1]{Sharafutdinov1994}.
\begin{lemma}
Let  $ f\in \sch(S^m)$ be a symmetric $m$-tensor field and  $g$, $v$ be as in the above Lemma \ref{th: decomposition of f in frequency variable}. Then for any multi-index $ \alpha $, the following identities hold:
		\begin{align}\label{estimate_for_g}
		D^{\alpha} g_{i_1i_2\dots i_m}(x)=  |x|^{-2(|\alpha|+m)}\sum\limits_{|\beta|\le |\alpha|} P^{\alpha j_1 \dots j_m}_{\beta i_1 \dots i_m}(x) D^{\beta} f_{j_1 \dots j_m}(x),
	\end{align}
	\begin{align}\label{esti_for_v}
	D^{\alpha} v_{i_1 \dots i_{m-k}}(x)=  |x|^{-2(|\alpha|+m)}\sum\limits_{|\beta|\le |\alpha|} Q^{\alpha j_1 \dots j_m}_{\beta i_1 \dots i_{m-k}}(x) D^{\beta} f_{j_1 \dots j_m}(x)
	\end{align}
	where $ P^{\alpha j_1 \dots j_m}_{\beta i_1 \dots i_m}(x) $ and $ Q^{\alpha j_1 \dots j_m}_{\beta i_1 \dots i_{m-k}}(x) $ are homogeneous polynomial of degree $ (2m+|\alpha|+|\beta|) $ and $ (2m+|\alpha|+|\beta|-k) $ respectively. Also, $ D=(D_1,\dots,D_n) $, $ D_{j}=-\I \PD_{x_j} $. 
\end{lemma}
\begin{proof}
Let us start with the observation that if we expand the right hand side of the expression for $g$ given in \eqref{eq: expression of g}, then every term in this expansion will be of the following form:
\begin{align*}
    \frac{x_{i_1}\dots x_{i_p}x^{j_1}\dots x^{j_p}}{|x|^{2p}}f_{j_1\dots j_pi_{p+1}\dots i_m} \quad \mbox{ for some }\quad 0 \leq p \leq m.
\end{align*}
Keeping this observation in mind, we will prove our lemma using induction on  $\alpha$.  For $ |\alpha|=1 $, we get $ D^\alpha= -\I \PD_{x_k} $ for some $\ 1\le k\le n $ and therefore 
\[ D^{\alpha }\left(\frac{x_{i_1}\dots x_{i_p}x^{j_1}\dots x^{j_p}}{|x|^{2p}}\right) =  \frac{\mbox{homogeneous poly of degree} (2p+1)}{|x|^{2(p+1)}}. \]
For $p =m$, the above equality becomes 
\[ D^{\alpha }\left(\frac{x_{i_1}\dots x_{i_m}x^{j_1}\dots x^{j_m}}{|x|^{2m}}\right) =  \frac{\mbox{homogeneous poly of degree} (2m+1)}{|x|^{2(m+1)}}.  \]
Using this, one can easily verify the following equality
 $$ D^{\alpha}g_{i_1i_2\cdots i_m} = \frac{1}{|x|^{2(m+1)}}\sum_{|\beta|=0}^{1} P^{\alpha j_1\dots j_m}_{\beta i_1\dots i_m}(x) D^{\beta} f_{j_1\dots j_m}(x) $$
 where $P^{\alpha j_1\dots j_m}_{\beta i_1\dots i_m}(x)$ is a homogeneous polynomial of degree $ 2m+1+|\beta|$. This shows that our result is true for $|\alpha| =1$. 
 
 Now,  assume that the result is true for $ |\alpha|=k$ then we aim to verify the result for $|\alpha| = k+1$. The idea here is to break $\alpha$ (such that $|\alpha| = k+1$) as $ \alpha= \gamma_1+\gamma_2 $ with $ |\gamma_1|=k $ and $ |\gamma_2| =1$. Then by applying $D^{\gamma_2}$ (which is same as the case $ |\alpha|=1$) to $ D^{\gamma_1}g_{i_1i_2\cdots i_m}$, we get the desired result for $g$.

 Finally to get the estimate for $v$, first we apply $ j_{x^{\otimes k}}$ to the equation \eqref{eq: decomposition of f in frequency variable} and then by again using a similar induction argument on $ \alpha $, we conclude the proof our lemma.
\end{proof}
Now, we are ready to present our decomposition theorem for symmetric $m$-tensor fields, which is one of the key aspects of this article and this decomposition will be used at several places later.  
		\begin{theorem}\label{th:decomp_in_original_sp}
		Let $ f \in \mathcal{S}(S^m)$ be a symmetric $m$-tensor field defined on $\Rb^n$ and  $ 1 \leq k \leq   \min\{n-1, m\}$ be a fixed positive integer. Then there exist uniquely determined smooth symmetric $m$-tensor field $ g $ and $(m-k)$-tensor field $ v $ satisfying  
			\begin{align}\label{decomp_of_f}
			f= g+\d^k v;  \quad \ \delta^{k} g=0,
			\end{align}	
$g(x), v(x) \rightarrow 0 \ \mbox{as}\ |x| \rightarrow \infty$. Additionally, we have the following decay estimates: 
	\begin{align}\label{estimate_for_g_and_v}
				|g(x)| \le C(1 + |x|)^{1-n}; \qquad 
			 |\d^{\ell}v(x)| \le C(1 + |x|)^{k+1-\ell} \quad \mbox{ (for } 0 \leq \ell\leq k).
			 			\end{align}	
The tensor fields $g$ and $v$ will be called the $k$-solenoidal part and the $k$-potential part of $f$ respectively.
		\end{theorem}
\begin{proof}[\textbf{Proof of existence}] 
We use the notation $ \widehat{f}(y) $ for the Fourier transform of $f$ which we define component-wise, that is, $$\widehat{f_{i_1\dots i_m}}(y) = \widehat{f}_{i_1\dots i_m}(y).$$
Then, we apply Theorem \ref{th: decomposition of f in frequency variable} to find unique symmetric tensor fields $\widehat{g}$ and $\widehat{v}$, of order $m$ and $(m-k)$ respectively, such that
     \begin{align}\label{1.12}
	\widehat{f }(y)=\widehat{g}(y) + i_{y^{\otimes k}}\widehat{v}(y) \ \ \mbox{and }\    j_{y^{\otimes k}} \widehat{g}(y)=0.	\end{align}
Using relations \eqref{estimate_for_g} and \eqref{esti_for_v}  for $\widehat{g}$ and $\widehat{v}$, we have that the both fields $ \widehat{g}(y) $ and $ \widehat{v}(y) $ are smooth on $ \Rn_{0}=\Rn\setminus\{0\} $, decay rapidly as $|y| \rightarrow \infty$. Additionally, we also have the following estimates  for $|y| \leq 1$ and for any multi-index $\alpha =(\alpha_1, \dots, \alpha_n)$:
		\begin{equation}\label{1.13}
		|D^{\alpha}\widehat{g}(y)|\le |y|^{-|\alpha|}, \qquad  	|D^{\alpha}\widehat{v}(y)|\le |y|^{-|\alpha|-k}.
		\end{equation}
From above estimates, we see that  $ D^{\alpha} \widehat{g}(y) $ is integrable for $ |\alpha|\le n-1 $ and $  D^{\alpha} \widehat{v}(y) $ is integrable for $ |\alpha| \le n-k-1 $. Hence $g$ and $v$ are smooth under the assumption $ 1 \leq k \leq \min\{m, n-1\}$. Also, by a direct application of the inverse Fourier transform to \eqref{1.12}, we get the following required decomposition:
		\begin{align*}
		f= g+\d^k v;  \quad \ \delta^{k} g=0.
		\end{align*}
Further, the summability  condition of $\widehat{g}$ and $\widehat{v}$ will give $g(x)$, $v(x) \rightarrow 0$ as $|x| \rightarrow \infty$. 
		Thus the only thing remains to show is the following estimates:
		$$				|g(x)| \le C(1 + |x|)^{1-n}; \qquad 
		|\d^\ell v(x)| \le C(1 + |x|)^{k+1-\ell-n} \quad \mbox{ (for } 0 \leq \ell\leq k).$$
		We show the estimate for $g$ in detail and estimate for $v$ can be achieved by similar arguments. We start by writing $ g $ in terms of Fourier inversion formula as follows:
		\begin{align*}
		g(x)&=\int_{\Rn} e^{\I x\cdot y} \widehat{g}(y) d y\\
		\Longrightarrow \qquad  x^{\alpha} g(x)&=  (-\I)^{|\alpha|}\int_{\Rn} \widehat{g}(y)D_{y}^{\alpha} e^{\I x\cdot y}\   d y.
		\end{align*} 
As $ \widehat{g}(y)$ is not smooth at the origin, in order to apply the integration by parts in the above identity, we rewrite the above integral in the following way:
		\begin{align*}
		x^{\alpha} g(x)&=  (-\I)^{|\alpha|}\lim_{\epsilon \rightarrow 0}\int_{|y|\geq \epsilon} \widehat{g}(y)D_{y}^{\alpha} e^{\I x\cdot y}\   d y\\
		&= \I^{|\alpha|}\lim_{\epsilon \rightarrow 0}\left( \int_{|y|\ge \epsilon} D_{y}^{\alpha}(\widehat{g}(y)) e^{\I x\cdot y}\ d y -\int_{|y|=\epsilon}D_{y}^{\alpha}(\widehat{g}(y)) \nu^{\alpha} e^{\I x\cdot y}\ d \sigma(y)\right).
		\end{align*}
Using inequality $ |D^{\alpha}\widehat{g}(y)|\le |y|^{-|\alpha|}$ from \eqref{1.13}, we conclude that $\lim\limits_{\epsilon \rightarrow 0} \int_{|y|=\epsilon}D_{y}^{\alpha}(\widehat{g}(y)) \nu^{\alpha} e^{\I x\cdot y}\ \d \sigma(y)  $ equals to $ 0 $ for $ |\alpha|\le n-2 $  and constant for $|\alpha| = n-1  $. Additionally, we have $ D^{\alpha} \widehat{g} \in L^1(\Rn) $ for $ |\alpha|\le n-1 $ gives
\begin{align*}
|x^{\alpha} g(x)|\le & C_{\alpha},
\end{align*}
where $ C_{\alpha} $
is a constant depending only on the multi-index $ \alpha $. Taking sum over $ |\alpha| $ from $ 0 $ to $ n-1$ and using the fact that $ (1+|x|)^{n-1}  $ and $ \sum\limits_{|\alpha|=0}^{n-1} |x^{\alpha}|$ are comparable, we get the estimate
\begin{align*}
|g(x)|\le C (1+|x|)^{1-n}. 
\end{align*}
Similar argument will work to derive the estimate of $ v $ and its derivatives. This finishes the proof of existence. \\
\noindent \textbf{\textit{Proof of uniqueness.}} Assume if possible, we have two such decomposition, that is, there are $g_1,\  g_2, \ v_1$ and $v_2$ satisfying 
\begin{align*}
    g_1 + \d^k v_1 =  f =     g_2 + \d^k v_2, \quad \mbox{ and } \quad  \delta^{k}g_1=0 =   \delta^{k}g_2\\
 \Rightarrow   (g_1 -g_2) + \d^k (v_1 -v_2) =  0, \quad \mbox{ and } \quad  \delta^{k}(g_1-g_2)=0.
\end{align*}
 \noindent Therefore to prove the uniqueness of the decomposition, it is enough to prove $ f=0 $ implies  $ g=v=0 $. Now $f=0$ gives $ g+\d^{k}v=0 $ and $ \delta^{k}g=0 $. Since $ g \in \mathcal{S}'(S^m) $ and $ v\in \mathcal{S}'(S^{m-k}) $, where $ \mathcal{S}' $ denotes the space of tempered distributions. Applying Fourier transform of the equations $ g+\d^{k}v=0 $ and $ \delta^{k}g=0 $, we get $ \widehat{g}(y)+ (\I)^{k} i_{y^{\otimes {k}}}\widehat{v}(y)=0 $ and $ j_{y^{\otimes (k)}} \widehat{g}(y)=0$. By Theorem \ref{th: decomposition of f in frequency variable} we have $ \widehat{g}(y)=\widehat{v}(y)=0 $ in $ \Rn_0= \Rn\setminus\{0\} $, $i.e., $ the support of distributions is contained in $ \{0\}.$  Thus  $ \widehat{g} $ and $ \widehat{v} $ can be written as finite linear combination of derivatives of Dirac delta distribution.
  Therefore 
 \[ \widehat{g} = \sum\limits_{|\alpha|\le p} c_{\alpha} \PD^{\alpha} \delta_{0}\] for some positive integer $ p $ and $ \delta_0 $ is the Dirac delta distribution. 
 
  Again $ \PD^{\alpha}\delta_0 \in \mathcal{S}'(\Rn)$ be the space of tempered distributions, for any multi-index $ \alpha $. Taking inverse Fourier transform of the above in the sense of tempered distributions, we obtain $ g $ is a polynomial of degree  almost $ p$. But $ g(x) \rightarrow 0$ as $ |x| \rightarrow \infty$  implies $ g = 0 $ in $ \Rn $.
One can argue similarly  and conclude that $ v(x) =0$ in $ \Rn $.     
\end{proof}

\begin{remark}
  We remark that, the estimates for the Fourier transform of $g$ and $v$ in \eqref{1.13} are optimal and can not be improved. 
\end{remark}

\section{Kernel description and Injectivity result for the operator $\Ic^k$}\label{sec: Injectivity results}
It is known \cite[Theorem 2.2.1]{Sharafutdinov1994} that the ray transform $\Ic^0/ I^0$  is injective over solenoidal tensor fields (sometimes also  called $I^0$ is s-injective) in $\Rb^n$. In a recent article \cite{Krishnan2018}, authors showed the injectivity of $\Ic^m$ over symmetric $m$-tensor fields in $\Rb^n$.  In this section, our aim is to generalize this injectivity result for $\Ic^k$ ($0 < k < m$). Additionally, we also provide an explicit description for the kernel of  $\Ic^k$ ($0 < k < m$).
\begin{theorem}[Injectivity of $\Ic^k$]\label{injectivity_result}
	Let $ f\in \mathcal{S}(S^m) $ be a $(k+1)$-solenoidal tensor field in $\Rn$, that is, $\delta^{k+1} f =0$. Then 
	$$\Ic^k f = 0 \qquad \Longrightarrow \qquad f \equiv 0.$$
		In other words, the operator $\Ic^k$ is injective over $(k+1)$-solenoidal tensor fields.
\end{theorem}
Given a symmetric $m$-tensor field, we define a  symmetric $(m -\ell)$- tensor field  $f_{m-\ell}$ obtained from $ f $ by fixing the first $\ell$ indices $ i_1,\dots,i_\ell$. This can be done by fixing any $\ell$ indices. Due to symmetry it is enough to fix the first $\ell$ indices that is,
	\begin{equation}\label{def: of restricted tensor field}
\left( f_{m-\ell}\right)_{j_1\dots j_{m-\ell}} = f_{i_1\dots i_\ell j_1\dots j_{m-\ell}}, \quad \mbox{ where } i_1,\dots, i_{\ell} \  \mbox{ are fixed.}
	\end{equation} 
Using this notation, the extended $q$-th integral moment ray transform of tensor field $f_{m-\ell}$ for any fixed choice of $i_1,\dots,i_\ell$ will be denoted by $J^q f_{m-\ell} (x, \xi)$, for any integer $ q \ge 0$.
\begin{lemma}\label{inversion}
	If $ I^0f,\dots,I^r\!f\,\  (0\le r\le m)$ are given for a symmetric $m$-tensor field $ f \in  \Sc(S^m)$. Then the following identity holds
	\begin{equation}\label{inversion_general}
	(J^0f_{m-r})_{i_1\dots i_r}= \frac{ (m-r)!}{m!}\sigma(i_1\dots i_r)\sum_{p=0}^{r} (-1)^p \binom{r}{p}\, \frac{\partial^r J^pf}{\partial x^{i_1}\dots\partial x^{i_p}\partial\xi^{i_{p+1}}\dots\partial\xi^{i_r}}
	\end{equation}
 for  $1\le i_1,\dots,i_r\le n $.
\end{lemma}
\begin{proof}
This result has been already proved in \cite[Theorem 3.1]{Krishnan2018} for the case $r=m$ and we follow similar technique to prove the result for case $(0 \leq r < m)$ with the required modifications. 

The idea is to use induction on $m$. For $m=0$, the only choice for $r$ is $0$ and hence the relation \eqref{inversion_general} holds trivially. In fact, if $r=0$ then the relation \eqref{inversion_general} holds for any $m$. Assume the relation \eqref{inversion_general} is true for $m$ tensor fields with $ 0 \le r < m$. We want to use this induction hypothesis to verify \eqref{inversion_general} for any  $1\le r+1<m+1$.
	
 Differentiating $ J^p\!f $ with respect to $ \xi^{i_{r+1}} $ we get  
	\begin{align*}
	J^p\!f_m&=\frac{1}{m+1}\left(\frac{\PD J^pf}{\PD \xi^{i_{r+1}}}-\frac{\PD J^{p+1}f}{\PD x^{i_{r+1}}}\right)\\=& \int_{-\infty}^{\infty} t^p \left( f_{i_1\cdots i_{m}i_{r+1}}(x+t\xi)\, \xi_{i_1}\,\cdots \xi_{i_{m}}\right)\d t
	\end{align*}
	for $ 0\le p\le r $ and $ f_m= f_{m+1-1}$ is a symmetric $ m $ tensor field given by \eqref{def: of restricted tensor field} . Thus by induction hypothesis, we have
	\begin{align*}\label{Eq2.2}
	(J^0(f_m)_{m-r})_{i_1\cdots i_r}&= \frac{ (m-r)!}{m!}\sigma(i_1\dots i_r)\sum_{p=0}^{r} (-1)^p \binom{r}{p}\, \frac{\partial^r J^pf_m}{\partial x^{i_1}\dots\partial x^{i_p}\partial\xi^{i_{p+1}}\dots\partial\xi^{i_r}}\nonumber\\
	&= \frac{ (m-r)!}{(m+1)!}\sigma(i_1\dots i_r)\sum_{p=0}^{r} (-1)^p \binom{r}{p}\, \frac{\partial^r }{\partial x^{i_1}\dots\partial x^{i_p}\partial\xi^{i_{p+1}}\dots\partial\xi^{i_r}} \left(\frac{\PD J^pf}{\PD \xi^{i_{r+1}}}-\frac{\PD J^{p+1}f}{\PD x^{i_{r+1}}}\right).
	\end{align*}	
	Since $(J^0(f_m)_{m-r})_{i_1\cdots i_r}= (J^0f_{m-r})_{i_1\cdots i_{r+1}} $, which is symmetric with respect to indices $ i_1 \dots i_{r+1}$. Therefore, above equation reduces to 
	\begin{equation}\label{Eq2.3}
	(J^0f_{m-r})_{i_1\cdots i_{r+1}}= \frac{ (m-r)!}{(m+1)!}\sigma(i_1\dots i_{r+1})\left[\sum_{p=0}^{r} (-1)^p \binom{r}{p}\, \frac{\partial^r }{\partial x^{i_1}\dots\partial x^{i_p}\partial\xi^{i_{p+1}}\dots\partial\xi^{i_r}} \left(\frac{\PD J^pf}{\PD \xi^{i_{r+1}}}-\frac{\PD J^{p+1}f}{\PD x^{i_{r+1}}}\right)\right].
	\end{equation}
	Using, the arguments used in \cite[Theorem 3.1]{Krishnan2018}, the term inside the bracket can be expressed as 
	\begin{align*}
	\sum_{p=0}^{r} (-1)^p \binom{r}{p}\, \frac{\partial^r }{\partial x^{i_1}\dots\partial x^{i_p}\partial\xi^{i_{p+1}}\dots\partial\xi^{i_r}} \left(\frac{\PD J^pf}{\PD \xi^{i_{r+1}}}-\frac{\PD J^{p+1}f}{\PD x^{i_{r+1}}}\right)&= \sum_{p=0}^{r+1} (-1)^p \binom{r+1}{p}\, \frac{\partial^{r+1}J^pf }{\partial x^{i_1}\dots\partial x^{i_p}\partial\xi^{i_{p+1}}\dots\partial\xi^{i_{r+1}}}.
	\end{align*}
	With the help of this, \eqref{Eq2.3} implies  
	\begin{equation*}
	(J^0f_{m-r})_{i_1\cdots i_{r+1}}=\frac{ (m-r)!}{(m+1)!}\sigma(i_1\dots i_{r+1}) \sum_{p=0}^{r+1} (-1)^p \binom{r+1}{p}\, \frac{\partial^{r+1}J^pf }{\partial x^{i_1}\dots\partial x^{i_p}\partial\xi^{i_{p+1}}\dots\partial\xi^{i_{r+1}}}.
	\end{equation*}
	This completes the proof.	
\end{proof}

\begin{proof}[Proof of Theorem \ref{injectivity_result}]
	Let $f$ be a symmetric $m$-tensor field in $\Rb^n$ satisfying $\delta^{k+1} f =0$ and $\Ic^k f  =0 $, that is,  $ I^\ell f =0 $ for $\ell = 0, 1 , \dots, k$. Our aim is to show that these conditions imply $f \equiv 0$. 
	

	Before moving further, recall  $ J^0f,\dots,J^k f $ are the extended operators satisfying $ J^{\ell}f|_{\tn} =I^{\ell}f$ for $ \ell=0,1,\dots,k$.  By Lemma \ref{inversion}, we have 
	\begin{equation*}
	J^0 f_{m-\ell}(x, \xi) =\frac{(m-\ell)!}{\ell!}\sigma(i_1\dots i_\ell)\sum\limits_{r=0}^{\ell}(-1)^r{\ell\choose r}\frac{\partial^{\ell}J^r f(x, \xi)}
	{\partial x^{i_1}\dots\partial x^{i_r}\partial\xi^{i_{r+1}}\dots\partial\xi^{i_\ell}}.
	\end{equation*}	
From equation \eqref{eq:relation between Ik and Jk}, we know  $  I^{\ell}f(x, \xi)= 0 $  implies $  J^{\ell}f (x, \xi)=0 $ for each $ \ell=0,1,\dots,k $. Therefore, the above equation gives 
$$ J^0 f_{m-\ell}(x,\xi)=0.$$ 
Taking the the Fourier transform of the above equation over  $\tn$ yields, see \cite[Equation 2.1.15]{Sharafutdinov1994},
\begin{align*}
	\left\langle \widehat{f}_{m-\ell}(y) , \xi^{m-\ell}\right\rangle  =0 \quad \mbox{for}\quad y\perp\xi. \quad  
	\end{align*}
Therefore for all $y \perp \xi$ we have
	\begin{align}\label{eq: Fourier transform of Jf(m-l)}
	\left\langle \widehat{f}(y), y^{\ell}\otimes \xi^{m-\ell}\right\rangle =0  \quad \mbox{for}\quad  \ell=0,1, \dots,k.
	\end{align}
	For a fixed  $y \in \Rb^n$, let $\zeta_1, \zeta_2,\dots,  \zeta_{n-1}$ be $(n-1)$ linearly independent vectors in the hyperplane $y^\perp$. Then, we can rewrite the above conditions as follows:
	\begin{align*}
\left\langle	\widehat{f}(y), y^{\ell}\otimes  \zeta_{i_1}^{\otimes j_1}\otimes \cdots \otimes \zeta_{i_{n-1}}^{\otimes j_{n-1}}\right\rangle = 0, \quad \mbox{where } 1 \leq i_1, \dots, i_{n-1} \leq (n-1)\ \mbox{ and } \sum_{p=1}^{n-1} j_p = m-\ell.
	\end{align*}	  
The collection $\left\{ y^{\ell}\otimes  \zeta_{i_1}^{\otimes j_1}\otimes \cdots \otimes \zeta_{i_{n-1}}^{\otimes j_{n-1}}\right\}$ is a linearly independent set, for details see \cite[Section 5]{Venky_and_Rohit}. Since $f$ is symmetric, the above relation provides $\begin{pmatrix}
	n+m-\ell-2\\
	m-\ell
	\end{pmatrix}$ independent conditions on $\widehat{f}(y)$ for every fixed $y \in \Rb^n$ and  $ 0 \leq \ell \leq k$. Therefore in total, we have 
	\begin{align*}
	\sum_{r =m-k}^{m}\begin{pmatrix}
	n+r-2\\
	r
	\end{pmatrix}
	\end{align*}
	independent conditions. But the dimension of a symmetric $m$-tensor in $\Rb^n$  is 
	$$\begin{pmatrix}
	n+m-1\\
	m
	\end{pmatrix} = \sum_{r =0}^{m}\begin{pmatrix}
	n+r-2\\
	r
	\end{pmatrix}. $$ Therefore, we require $\sum_{r =0}^{m-k-1}\begin{pmatrix}
	n+r-2\\
	r
	\end{pmatrix}$  more condition on $\widehat{f}(y)$ for the unique recovery of $\widehat{f}$ at $y \in \Rb^n$. To obtain these relations, we use the condition $\delta^{k+1}f=0$. By similar argument, we again take the Fourier transform of  $\delta^{k+1}f=0$ to get 
	\[ \left\langle\widehat{f}(y), y^{k+1}\right\rangle =0.\]
	This is a symmetric $(m-k-1)$-tensor field.  Taking the tensor product with $y^{p-1}\otimes\xi^{m-k-p} \,$  for  $ 1\le p \le m-k $,  entails
	\begin{align*}
	\left\langle\widehat{f}(y),y^{k+p}\otimes\xi^{m-k-p}\right\rangle =0.
	\end{align*}
	We can argue in exactly similar way as we did above to conclude that the above equality will provide total $\sum_{r =0}^{m-k-1}\begin{pmatrix}
	n+r-2\\
	r
	\end{pmatrix}$ independent conditions which are also independent of the conditions we get from \eqref{eq: Fourier transform of Jf(m-l)}.

	Thus by combining all these independent conditions, we get $ \hat{f}(y)=0 $ for all  $ y\ne 0 $ $ i.e., $ support of components of $ \widehat{f}\subseteq \{0\} $. Therefore components of $ \hat{f}(y) $ are a distribution which can written as a linear combination of derivatives of the Dirac delta distribution. But the condition $ f\in \mathcal{S}(S^m) $ implies $ f=0 $. 
\end{proof}

\begin{theorem}[Kernel of $\Ic^k$]\label{th:injectivity result}
A symmetric $m$-tensor field $ f \in \mathcal{S}(S^m)$ is in the kernel of the operator $\Ic^k$ for  $ 1 \leq k \leq \min\{m, n-1\}$ if and only if $ f =  \d^{k+1} v$, for some  $(m-k-1)$-tensor field $v$ satisfying  $\d^{\ell}v \rightarrow 0 $ as $ |x| \rightarrow \infty$.
\end{theorem}
\begin{proof}
    To proof the `if' part of the theorem, assume  $ f=\d^{k+1} v$ for some $ v \in C^{\infty}(S^{m-k-1})$ satisfying   $\d^{\ell}v \rightarrow 0 $ as $ |x| \rightarrow \infty$ for $ \ell=0,1,\dots,k $. Then a simple application of integration by parts entails 
    $$ I^{\ell}(f)= I^0(\d^{k+1-\ell}v)=0, \qquad \mbox{ for } 0\le \ell \le k.$$ Conversely, if $ f\in \mathcal{S}(S^{m}) $ satisfies $ I^{\ell}f =0,\, \ell=0,1,\dots,k$. According to our decomposition theorem  \ref{th:decomp_in_original_sp},  $ f $ can be written as 
	\[f=g+\d^{k+1}v,\quad \delta^{k+1} g=0 \quad \mbox{and}\quad  \d^{\ell}v\, \rightarrow 0 \quad \mbox{as} \quad |x| \rightarrow \infty, \ \  0\le \ell\le k.\]
 	\noindent Now from Lemma \ref{inversion} we have  \[ J^{0}f_{m-\ell}(x,\xi)=0 \quad \ell=0,1,\dots,k \]  where $ f_{m-\ell} $ is symmetric $ m-\ell $ tensor field obtained from $ f $ by fixing $ \ell $ indices. This   imply $ I^{0}f_{m-\ell}= J^{0}f_{m-\ell}|_{\tn}=0.$   Fourier transform of $ I^{0}f_{m-\ell} $ gives   \[  \widehat{I^{0}f}(y,\xi)=\left\langle\widehat{f}_{m-\ell}(y), \xi^{m-\ell} \right\rangle=0 \] for $ y\perp \xi $ and $ 0\le \ell\le k $. 
	This, $ y\perp\xi $ and  together with the fact $ \widehat{f}(y) = \widehat{g}(y)+ y^{k+1}\otimes \widehat{v}(y)$ gives 
 	\[\left\langle \widehat{g}_{m-\ell}(y),\xi^{m-\ell} \right\rangle=0. \]
 	By multiplying this equation with $ y^{\ell}=\underbrace{y\otimes\cdots\otimes y}_{\ell\, \mbox{times}}\,(y\ne 0) $ and then summing over $ \ell $ indices , we obtain
		\[\left\langle\widehat{g}(y), y^{\ell}\otimes \xi^{m-\ell}\right\rangle=0\quad \mbox{for}\quad 0\le \ell \le k. \]
		Applying Fourier transform on the equation $\delta^{k+1}g=0$, we get
 	 \[\left\langle\widehat{g}(y)\otimes y^{k+1}\right\rangle=0.\]
 	 This is a symmetric $(m-k-1)$-tensor field and taking the tensor product with $y^{r-1}\otimes\xi^{m-k-r} \,$  for  $ 1\le r \le m-k $  entails
 	 \begin{align*}
	\left\langle \widehat{g}(y), y^{k+r}\otimes\xi^{m-k-r}\right\rangle =0.
 	 \end{align*}
 	Thus for a  non-zero vector $ y\in \Rn $ with $ y\perp\xi $ , we have
\begin{align}\label{Eq3.11}
 		\left\langle \widehat{g}(y), y^{r}\otimes\xi^{m-r} \right\rangle=0
 	\end{align}
 	for $ 0\le r\le m $.
 	Therefore we are in the same situation as in the Theorem \ref{injectivity_result} and  have enough linearly independent relations, which implies $ \widehat{g}(y)=0 $ for $ y\ne 0 $.
 	 Since $ \widehat{g}(y) $ is an integrable function. So we can view this as a distribution and support of $ (\widehat{g})\subseteq \{0\} $. Amending the arguments used in the proof of uniqueness part of the Theorem \ref{th:decomp_in_original_sp}, we get $ g=0 $ in $ \Rn $.  
 	 
 	 Putting $g=0$ in the decomposition above, we achieve $f =\d^{k+1}v$ which completes the proof of converse part as well.
\end{proof}
\section{Range characterization}\label{sec:range characterization}
\noindent This section is devoted to a detailed description of the range for the operator $\Ic^k$. More specifically, we prove
	\begin{theorem}\label{th:range characterisation for I-k}
		Let $ n\ge 3$ and $1 \leq k \leq m$. An element $ (\vf^0,\vf^1,\dots,\vf^k)\in (\Sc(\tn))^{k+1}$ belongs to the range of operator $\Ic^k$ if and only if the following two conditions are satisfied:
		\begin{enumerate}
		    \item 	$\vf^\ell(x,-\xi)= (-1)^{m-\ell}\vf^\ell(x,\xi)  $ for $ \ell=0,1,\dots,k$.
		    \item  For $0\le \ell\le k$,	the functions $ \psi^{\ell} \in C^{\infty}(\Rn \times \Rn \setminus\{0\})$, defined by
		    \begin{equation}\label{def:of psi l}
		      \psi^{\ell}= |\xi|^{m-2\ell-1}\sum\limits_{r=0}^\ell(-1)^{\ell-r}\binom{\ell}{r}|\xi|^r \langle\xi,x\rangle^{\ell-r}\,(I^r\!f)
	\left(x-\frac{\langle x, \xi \rangle}{|\xi|^2}\xi,\frac{\xi}{|\xi|}\right)
		    \end{equation}
	satisfies the equations
		\begin{equation}\label{Johns's condition for psi k}
			\Big(\frac{\partial^2}{\partial x^{i_1}\partial\xi^{j_1}}-\frac{\partial^2}{\partial x^{j_1}\partial\xi^{i_1}}\Big)\dots
			\Big(\frac{\partial^2}{\partial x^{i_{m+1}}\partial\xi^{j_{m+1}}}-\frac{\partial^2}{\partial x^{j_{m+1}}\partial\xi^{i_{m+1}}}\Big) 	\psi^k=0
	\end{equation}
		for all indices $1\leq i_1,j_1,\dots,i_{m+1},j_{m+1}\leq n$.
	\end{enumerate}
	\end{theorem}
\noindent The range of the operator $ \Ic^m$ was already proved in \cite{Krishnan2019a}. Therefore we consider the $ k<m $ case here. 
	
The following theorem from \cite[Theorem 2.10.1]{Sharafutdinov1994} provides the range characterization for the operator $I^0$ and we will use this repeatedly to prove our range characterization theorem for $\Ic^k$.
\begin{theorem} \label{Th3.1}
		Let $ n\ge 3 $. A function $\varphi\in{\mathcal S}(T\Sb^{n-1})$ belongs to the range of $I^0$ if and only if $\varphi$ satisfies the following two conditions:
		\begin{enumerate}
			\item[(1)] $\varphi(x,-\xi)=(-1)^m\varphi(x,\xi)$;
			\item[(2)] the function $\psi\in C^\infty\big({\mathbb R}^n\times({\mathbb R}^n\setminus\{0\})\big)$, defined by \begin{align*}
			    \psi(x,\xi)=|\xi|^{m-1}\varphi\Big(x-\frac{\langle\xi,x\rangle}{|\xi|^2}\xi,\frac{\xi}{|\xi|}\Big),
			\end{align*} satisfies the
			equations
			\begin{equation}
				\Big(\frac{\partial^2}{\partial x^{i_1}\partial\xi^{j_1}}-\frac{\partial^2}{\partial x^{j_1}\partial\xi^{i_1}}\Big)\dots
				\Big(\frac{\partial^2}{\partial x^{i_{m+1}}\partial\xi^{j_{m+1}}}-\frac{\partial^2}{\partial x^{j_{m+1}}\partial\xi^{i_{m+1}}}\Big) \psi=0
			\label{Eq1.9}
			\end{equation}
			for all indices $1\leq i_1,j_1,\dots,i_{m+1},j_{m+1}\leq n$.
		\end{enumerate}
	\end{theorem}
\noindent With the help of this John's operator, we rewrite the relation \eqref{Johns's condition for psi k} as follows:
$$	J_{i_{m+1}j_{m+1}}\dots J_{i_2j_2} J_{i_1j_1} \psi^k=0 \quad \mbox{for all indices} \quad 1\le i_1,j_1,\dots,i_{m+1},j_{m+1}\le n.$$
\subsection{Required lemmas and results for Proof of Theorem \ref{th:range characterisation for I-k}}\label{subsec:lemmas for range characterization}
We need a good amount of preparation before we get in to the proof of Theorem \ref{th:range characterisation for I-k}.  We start by making a quick observation that if a symmetric $m$-tensor field $ f\in \mathcal{S}(S^m) $  is given by   
		\begin{equation}\label{def:of f}
		f= \sum_{s=0}^{k}\d^{s} g_{s}, \quad \mbox{ where } g_s \in \Sc(S^{m-s}), \mbox{ for } s = 0, 1, \dots , k.
		\end{equation} 
Then using the identity  $I^{\ell}(\d f) =-\ell\, I^{\ell-1}\!f$ recursively, we get 
		\begin{align}
		I^{\ell} (\d^sg_s)&= \begin{cases} (-1)^s\,
		\binom{\ell}{s}\, s! \,I^{\ell-s} g_s& \mbox{if}\quad s\le \ell\\
		0& \mbox{if} \quad s> \ell.
		\end{cases}\nonumber \\
		\Longrightarrow \qquad  \vf^\ell =I^\ell\!f &= \sum_{s=0}^{\ell} (-1)^s\,
		\binom{\ell}{s}\ s! \,I^{\ell-s} g_s \label{eq:relation between phi and gs}\\ \mbox{ and } \qquad     \psi^\ell =J^\ell\!f &= \sum_{s=0}^{\ell} (-1)^s\,
		\binom{\ell}{s}\ s! \,J^{\ell-s} g_s,  \qquad \mbox{ for }  0\le \ell\le k \label{eq:relation between psi and gs}.
		\end{align}
Note, if we can find tensor fields $g_s$, for $0 \leq s \leq k$, satisfying the above relation \eqref{eq:relation between phi and gs} then $(\vf^0, \vf^1, \dots, \vf^k)$ will be in the range of operator $\Ic^k$ and $\Ic^k f = (\vf^0, \vf^1, \dots, \vf^k)$, where $f$ is given by equation \eqref{def:of f}. Keeping this key conclusion in mind, we present a series of lemmas essential to proceed further.
%
			\begin{lemma}\label{Lm5.1}
If $\psi^s$,\, for\, $ 0\le s\le \ell
	-1 $ is given by relation \eqref{eq:relation between psi and gs} for known tensor fields $g_s$ (for $0 \leq s \leq \ell-1$), then the function $ \chi^{\ell} \in C^{\infty}(\Rn \times\Rn\setminus{0})$ (for each fixed $0\leq  \ell \le k$) defined by
	  \begin{equation}\label{eq:definition_of_chi_l}
	  \chi^{\ell}=\frac{(-1)^\ell}{\ell!} \left(\psi^\ell - \sum_{s=0}^{\ell-1}(-1)^{s}\,\binom{\ell}{s}\,s! \,J^{\ell-s} g_{s}\right)
	  \end{equation}
	   satisfy the following properties:
	   \begin{enumerate}
	   \item For $(x, \xi) \in C^{\infty}(\Rn \times\Rn\setminus{0})$ and $t \in \Rb$, 
\begin{equation}\label{translation of chi l}
	\chi^\ell (x+t\xi,\xi)= \chi^{\ell}(x,\xi).
\end{equation}
	\item For $(x, \xi) \in C^{\infty}(\Rn \times\Rn\setminus{0})$ and $0 \neq t \in \Rb$,
\begin{equation}\label{homogenety w r to xi}
\chi^\ell (x,t\xi)=\frac{t^{m-\ell}}{|t|} \chi^{\ell}(x,\xi).
\end{equation}
		   \end{enumerate}
		   	\end{lemma}
\begin{proof}
For any $ t\in \mathbb{R} $, from \cite[Statement 2.8]{Krishnan2019a}, we have
\begin{align}\label{Eq5.10}
	\psi^{\ell}(x+t\xi,\xi)&= \sum_{p=0}^{\ell}\binom{\ell}{p}(-t)^{\ell-p} \psi^p(x,\xi) \nonumber \\
	&= \psi^{\ell}+\sum_{p=0}^{\ell-1}\sum_{s=0}^{p} \binom{\ell}{p}(-t)^{\ell-p}
	(-1)^{s}\binom{p}{s}s!\,J^{p-s} g_{s}, \qquad \mbox{ from \eqref{eq:relation between psi and gs}} \nonumber  \\
	&= \psi^{\ell}
	+\sum_{s=0}^{\ell-1}\sum_{p=s}^{\ell-1} \binom{\ell}{p}(-t)^{\ell-p}
	(-1)^{s}\binom{p}{s}s!\,J^{p-s} g_{s}.
\end{align}
Also, from definition of $J^\ell$, we get
\begin{align*}
J^{\ell}f(x+t\xi,\xi)= \sum_{p=0}^{\ell}\binom{\ell}{p}(-t)^{\ell-p} J^pf(x,\xi).
\end{align*}
By replacing $f$ by $g_s$ and $\ell$ by $\ell -s$, this relation reduces to 
\begin{align*}
J^{\ell-s}g_s(x+t\xi,\xi)= \sum_{p=0}^{\ell-s}\binom{\ell-s}{p}(-t)^{\ell-s-p} J^pg_s(x,\xi).	
	\end{align*}
Consider, 
\begin{align}
\sum_{s=0}^{\ell-1}(-1)^{s}\,\binom{\ell}{s}\,s! \,J^{\ell-s} g_{s}(x+t\xi,\xi)&= \sum_{s=0}^{\ell-1}\sum_{p=0}^{\ell-s}(-1)^{s}\,\binom{\ell}{s}\,s!\binom{\ell-s}{p}(-t)^{\ell-s-p} J^pg_s(x,\xi)\nonumber \\
&=\sum_{s=0}^{\ell-1}\sum_{p=0}^{\ell-s}(-1)^{s}\,\frac{\ell!}{p!\, (\ell-s-p)!}(-t)^{\ell-s-p} J^pg_s(x,\xi)\nonumber \\
&=\sum_{s=0}^{\ell-1}\sum_{p=s}^{\ell}(-1)^{s}\,\frac{\ell!}{(\ell-p)!\, (p-s)!}(-t)^{\ell-p} J^{p-s}g_s(x,\xi) \nonumber \\
&= \sum_{s=0}^{\ell-1}\sum_{p=s}^{\ell}(-1)^s (-t)^{\ell-p} \binom{\ell}{p}\binom{p}{s}\, s!  J^{p-s}g_s(x,\xi) \nonumber \\
&= \sum_{s=0}^{\ell-1}\sum_{p=s}^{\ell-1}(-1)^s (-t)^{\ell-p} \binom{\ell}{p}\binom{p}{s}\, s!  J^{p-s}g_s(x,\xi) \nonumber  \\&\qquad \quad + \sum_{s=0}^{\ell-1}(-1)^s  \binom{\ell}{s}\, s!\,  J^{\ell-s}g_s(x,\xi) \label{Eq5.11}.
\end{align}
Putting these expressions in the definition of function $\chi^\ell$ (see equation \eqref{eq:definition_of_chi_l}), we get  
\begin{equation*}
\chi^{\ell}(x+t\xi,\xi)= \frac{(-1)^\ell}{\ell!} \left(\psi^{\ell}(x+t\xi,\xi) - \sum_{s=0}^{\ell-1}(-1)^{s}\,\binom{\ell}{s}\,s! \,J^{\ell-s} g_{s}(x+t\xi,\xi)\right).
\end{equation*}
The identities \eqref{Eq5.10} and \eqref{Eq5.11} proved above yields
$$
\chi^{\ell}(x+t\xi,\xi)= \frac{(-1)^\ell}{\ell!} \left( \psi^{\ell}(x,\xi)- \sum_{s=0}^{\ell-1}(-1)^s  \binom{\ell}{s}\, s!\,  J^{\ell-s}g_s(x,\xi)\right)
= \chi^{\ell}(x,\xi).
$$
This completes the proof of identity \eqref{translation of chi l}. Next for $t\ne 0$, the definition of $\chi^\ell$ (equation \eqref{eq:definition_of_chi_l}) gives
\begin{align*}
	\chi^{\ell}(x,t\xi)= \frac{(-1)^\ell}{\ell!} \left( \psi^{\ell}(x,t\xi)- \sum_{s=0}^{\ell-1}(-1)^s  \binom{\ell}{s}\, s!\,  J^{\ell-s}g_s(x,t\xi)\right).
\end{align*} 
Then the required relation \eqref{homogenety w r to xi} can be achieved directly from the following two known homogeneity properties (first identity follows from direct computation and the second one is from \cite[Statement 2.8]{Krishnan2019a}) :
\begin{align*}
J^{\ell-s}g_s(x,t\xi)= \frac{t^{m-\ell}}{|t|} J^{\ell-s}g_s(x,\xi)\quad \mbox{ and } \quad 
 \psi^{\ell}(x,t\xi) = \frac{t^{m-\ell}}{|t|}\psi^{\ell}(x,\xi).
\end{align*}
Hence the proof of lemma is complete.
\end{proof}
	\begin{lemma}\label{extension lemma}
Let $\chi^\ell$ and $\psi^\ell$ satisfy same conditions as in previous lemma. Also, define the function $\widetilde{\chi}^\ell$ on $\tn$ by 
		\begin{equation*}
		\widetilde{\chi}^{\ell}= \frac{(-1)^\ell}{\ell!} \left(\vf^\ell - \sum_{s=0}^{\ell-1}(-1)^{s}\binom{\ell}{s}\,s!\, I^{\ell-s} g_{s}\right).
		\end{equation*}
Then $\widetilde{\chi}^\ell = \chi^\ell|_{\tn}$ and we can obtain $\chi^\ell$ from $\widetilde{\chi}^\ell$  using the following explicit relation:
		\[\chi^{\ell}(x,\xi)= |\xi|^{m-\ell-1}\,\widetilde{\chi}^{\ell}\left( x-\frac{\l x,\xi \r }{|\xi|^2}\xi, \frac{\xi}{|\xi|}   \right), \qquad \
		(x, \xi) \in \Rb^n \times \Rb^n\setminus \{0\}.\]
\end{lemma}
	\begin{proof}
For any $ t,s\in\mathbb{R} $ with $ s\ne 0 $ equations \eqref{translation of chi l} and \eqref{homogenety w r to xi} gives 
\begin{align*}
	\chi^{\ell}(x+t\xi,s\xi)=\frac{s^{m-\ell}}{|s|}\chi^{\ell}(x,\xi).
\end{align*}
Now choosing $ t= -  \frac{\l x,\xi\r}{|\xi|^2}$ and $ s = \frac{1}{|\xi|} $, this gives
\begin{align*}
\chi^{\ell}\left( x-\frac{\l x,\xi \r }{|\xi|^2}\xi, \frac{\xi}{|\xi|}   \right)= \frac{1}{|\xi|^{m-\ell-1}} \chi^{\ell}(x,\xi).	
\end{align*}
Therefore
\begin{align*}
\chi^{\ell}(x,\xi)&= |\xi|^{m-\ell-1}\chi^{\ell}\left( x-\frac{\l x,\xi \r }{|\xi|^2}\xi, \frac{\xi}{|\xi|}   \right)\\
&= |\xi|^{m-\ell-1}\widetilde{\chi^{\ell}}\left( x-\frac{\l x,\xi \r }{|\xi|^2}\xi, \frac{\xi}{|\xi|}   \right).
\end{align*}
	\end{proof}
The next three lemmas are direct adaptation of results from \cite{Sharafutdinov1994} and \cite{Krishnan2019a} and hence  we state  without giving their proofs.
\begin{lemma}\cite[Theorem 2.10]{Sharafutdinov1994}\label{general johns condition} 
	For every indices $ 1\le i_1,\dots,i_r\le n $ and each  $ h\in \mathcal{S}({S^r(\Rn)}) $, the next equality holds
	\begin{equation}
	J_{i_{r+1}j_{r+1}} \cdots J_{i_1j_1} J^0h=0.
	\end{equation}
    \end{lemma} 
\begin{lemma}\cite{Krishnan2019a} \label{Lm 4.1}
	For every $\ell=0,1,\dots,k$ and for every integer $k\ge 0$, the following equality holds:
	\begin{equation}\label{Eq4.1}
	\begin{aligned}
	\l\xi,\partial_x\r^\ell\psi^k=\begin{cases}
	(-1)^\ell\,{k\choose \ell}\,\ell!\,\psi^{k-\ell} &\mbox{if}\quad\ell\le k,\\
	0 &\mbox{if}\quad\ell>k.
	\end{cases}
	\end{aligned}
	\end{equation}
\end{lemma}
\begin{lemma}\cite[Lemma 2.6]{Krishnan2019a} \label{Lm 4.2}
	Let a function $\psi\in C^\infty\big({\Rn}\times{\Rn}\setminus\{0\})\big)$ be positively homogeneous in the second argument
	\begin{equation}
	\psi(x,t\xi)=t^\lambda\psi(x,\xi)\quad(t>0).
	\label{Eq2.10}
	\end{equation}
	Assume the restriction $\psi|_{T{\mathbb S}^{n-1}}$ to belong to ${\mathcal S}(T{\mathbb S}^{n-1})$. Assume also that restrictions to ${\mathcal S}(T{\mathbb S}^{n-1})$ of the function $\l\xi,\partial_x\r\psi$ and of all its derivatives belong to ${\mathcal S}(T{\mathbb S}^{n-1})$, i.e.,
	\begin{equation}
	\left.\frac{\partial^{k+p}(\l\xi,\partial_x\r\psi)}{\partial x^{i_1}\dots\partial x^{i_k}\partial \xi^{j_1}\dots\partial \xi^{j_p}}\right|_{T{\mathbb S}^{n-1}}\in{\mathcal S}(T{\mathbb S}^{n-1})\quad\mbox{for all}\quad 1\le i_1,\dots,i_k,j_1,\dots,j_p\le n.
	\end{equation}
	Then the restriction to ${\mathcal S}(T{\mathbb S}^{n-1})$ of every derivative of $\psi$ also belongs to ${\mathcal S}(T{\mathbb S}^{n-1})$, i.e.,
	\begin{equation}
	\left.\frac{\partial^{k+p}\psi}{\partial x^{i_1}\dots\partial x^{i_k}\partial \xi^{j_1}\dots\partial \xi^{j_p}}\right|_{T{\mathbb S}^{n-1}}
	\in{\mathcal S}(T{\mathbb S}^{n-1})\quad\mbox{for all}\quad 1\le i_1,\dots,i_k,j_1,\dots,j_p\le n.
	\end{equation}
\end{lemma}	
	
	\begin{lemma}\label{John's condition }
		Let $ \psi^{k} $ satisfies 
			\begin{equation}\label{Eq36}
			J_{i_1j_1} \cdots J_{i_{m+1}j_{m+1}}\psi^k=0
			\end{equation}
	and \begin{equation}\label{Eq37}
	\psi^{k}(x,t\xi)= \frac{t^{m-k}}{|t|}\psi^{k} (x,\xi)
	\end{equation}		
for each $ 0\le \ell\le k <m $.			
				For each indices $ 1\le i_1,i_2,\dots,i_{\ell}\le n $, let $ \Psi_{i_1\cdots i_\ell},\ 0\le \ell \le k<m$ denotes the function defined by 
				\begin{equation}\label{eq:John's condition}
				\Psi_{i_1\dots i_\ell}= \frac{ (m-\ell)!}{m!}\sigma(i_1\dots i_\ell) \Bigg(\sum\limits_{p=0}^{\ell}(-1)^p\binom{\ell}{p}\,\frac{\partial^\ell \psi^p}{\partial x^{i_1}\dots\partial x^{i_p}\partial\xi^{i_{p+1}}\dots\partial\xi^{i_\ell}}\Bigg).
				\end{equation}
			Then $ \Psi_{i_1\cdots i_\ell},\ 0\le \ell \le k<m$ satisfies the following John's condition
				\begin{equation}\label{Eq5.6}
				J_{i_1j_1} \cdots J_{i_{m-\ell+1}j_{m-\ell+1}} \Psi_{i_1\dots i_{\ell}}=0.
				\end{equation}
				\end{lemma}
				\begin{remark}
						The proof of this lemma is very similar to  \cite[Lemma 2.7]{Krishnan2019a}. So we do not present the complete proof and only indicate the key arguments here. We should mention that Lamma \ref{John's condition } is new and has not been proved in the earlier work \cite{Krishnan2019a}. 
				\end{remark}

			\begin{proof}
				According to Lemma \ref{Lm 4.1}, we have
			for every $0\le  r\le k $
			\begin{equation}\label{Eq40}
			\begin{aligned}
			\l\xi,\partial_x\r^{k-r}\psi^{k}=				(-1)^{k-r}\,{k\choose r}\,(k-r)!\,\psi^{r}.
			\end{aligned}
			\end{equation}
	Using above relation with  $k= \ell$ and $r= p$  in equation \eqref{Eq5.6},  we obtain
			\begin{align}\label{modified johns condition}
			J_{i_1j_1} \cdots J_{i_{m-\ell+1}j_{m-\ell+1}} \sigma(i_1\dots i_\ell) \Bigg(\sum\limits_{p=0}^{\ell}\frac{1}{(\ell-p)!}\,\frac{\partial^\ell \l\xi,\PD_x\r^{\ell-p}}{\partial x^{i_1}\dots\partial x^{i_p}\partial\xi^{i_{p+1}}\dots\partial\xi^{i_l}}\Bigg) \psi^{\ell}=0.
			\end{align} 	 
			Thus proving \eqref{Eq5.6} is equivalent to prove \eqref{modified johns condition}. 
			
			The proof will be based on induction argument on $ m $. For $m =0$, we have $ \ell=k=0 $ and it is easy to see that the equation  \eqref{modified johns condition} holds. Assume that the relation \eqref{modified johns condition} is true for some $ m $ with $k < m$ and for all $0 \le  \ell \le k $. Then, we aim to verify the result for $m+1$  with $ 1\le \ell +1\le k+1<m+1 $. 

For every index $i_{\ell+1}$ satisfying $1\le i_{\ell+1}\le n$, we define the function $\psi^{\ell}_{i_{\ell+1}}$ by
			\begin{equation*}
			\psi^{\ell}_{i_{\ell+1}}(x,\xi)=\left(\frac{1}{\ell+1}\,\frac{\PD}{\PD\xi^{i_{\ell+1}}} \l \xi,\PD_{x}\r + \frac{\PD}{\PD x^{i_{\ell+1}}} \right) \psi^{\ell+1}(x,\xi).
			\label{Eq2.36}
			\end{equation*}
			In the light of the homogeneity relation\eqref{Eq37} together with equation \eqref{Eq40},  the equation reduces to
			\begin{equation*}
			\psi^{\ell}_{i_{\ell+1}}(x,t\xi)=\frac{t^{m-\ell}}{|t|}\psi^{\ell}_{i_{\ell+1}}(x,\xi)\quad(0\neq t\in{\mathbb{R}}).
			\end{equation*}	
Then, we have from  \cite[Lemma 2.7]{Krishnan2019a}	  $\psi^{\ell}_{i_{\ell+1}}$ satisfies \eqref{Eq36}.		
Thus for every ${i_{\ell+1}}$ and for each $ 0\le \ell\le k<m $, the function $\psi^{\ell}_{i_{\ell+1}}$ satisfies hypotheses of Lemma \ref{John's condition }. By the induction hypothesis \eqref{modified johns condition}, for all $1\le i,j,{i_{\ell+1}},i_1,\dots, i_{\ell}\le n$,
			\begin{align*}
			\sigma(i_1\dots i_\ell) \Bigg(\sum\limits_{p=0}^{\ell}\frac{1}{(\ell-p)!}\,\frac{\partial^\ell \l\xi,\PD_x\r^{\ell-p}}{\partial x^{i_1}\dots\partial x^{i_p}\partial\xi^{i_{p+1}}\dots\partial\xi^{i_l}}\Bigg)J_{i_1j_1} \cdots J_{i_{m-\ell+1}j_{m-\ell+1}} \psi^{\ell}_{i_{\ell+1}}=0.
			\end{align*}
			Now using the similar arguments used in \cite[Lemma 2.7]{Krishnan2019a} we can conclude that
			\[ J_{i_1j_1} \cdots J_{i_{m-\ell+1}j_{m-\ell+1}}  \sigma(i_1\dots i_{\ell+1}) \Bigg(\sum\limits_{p=0}^{\ell+1}\frac{1}{(\ell+1-p)!}\,\frac{\partial^{\ell+1} \l\xi,\PD_x\r^{\ell+1-p}}{\partial x^{i_1}\dots\partial x^{i_p}\partial\xi^{i_{p+1}}\dots\partial\xi^{i_{\ell+1}}}\Bigg) \psi^{\ell+1}=0.    \]
			Thus  knowing \eqref{modified johns condition} for some $ m $ with $ 0\le \ell\le k<m $,  with the help of induction we are proving \eqref{modified johns condition} for $ m+1 $ with $ 1\le \ell+1\le k+1<m+1 $. Thus for $ m+1 $, $ \ell=0 $ case is still left. $ \ell=0 $  case  follows from the fact that, $ \l \xi,\PD_x\r $ commutes with John's operator $ J_{ij} $, $\l\xi,\partial_x\r^{k}\psi^{k}=				(-1)^{k}\,(k)!\,\psi^{0}  $ and \eqref{Eq36}.
	This completes the proof.
			\end{proof}
\begin{lemma}\label{Lm 8.2}
For $ 0\le p\le \ell-1 $, if $ \psi^p $ is given by \eqref{eq:relation between psi and gs}, 
then
		\begin{equation}
		\begin{aligned}
		\Psi_{i_1\dots i_\ell}=  \frac{1}{\binom{m}{\ell}}\frac{\PD^{\ell}}{\PD x^{i_1}\cdots x^{i_{\ell}}} \chi^{\ell} + \frac{1}{\binom{m}{\ell}} \sigma(i_1\dots i_{\ell}) \sum_{s=0}^{\ell-1} \binom{m-s}{\ell-s}\, \frac{\PD^{s} ((J^0g_s)_{m-\ell})_{i_1\dots i_{\ell-s}}}{\PD x^{i_{\ell-s+1}}\cdots \PD x^{i_{\ell}}}
		\end{aligned}
		\end{equation}
		where $ \Psi_{i_1\dots i_\ell} $ is given by \eqref{eq:John's condition}.  
	\end{lemma}
	\begin{proof}
To prove this lemma, we need to compute the term in parenthesis on the right hand side of equation \eqref{eq:John's condition} using the expression for $\psi^\ell =  \sum_{s=0}^{p}(-1)^{s}\binom{p}{s}\, s!\,  J^{p-s} g_{s}$. Consider
\begin{align*}
 	\sum_{p=0}^{\ell -1}(-1)^p\binom{\ell}{p}\frac{\partial^\ell \psi^p}{\partial x^{i_1}\dots\partial x^{i_p}\partial\xi^{i_{p+1}}\dots\partial\xi^{i_\ell}}&=\sum_{p=0}^{\ell -1} \sum_{s=0}^{p}(-1)^{p-s}\binom{p}{s} \binom{\ell}{p}\, s!\, \frac{\partial^\ell J^{p-s}g_s }{\partial x^{i_1}\dots\partial x^{i_p}\partial\xi^{i_{p+1}}\dots\partial\xi^{i_\ell}}\\
 	&=\sum_{s=0}^{\ell -1} \underbrace{\sum_{p=s}^{\ell-1}(-1)^{p-s}\binom{p}{s}\binom{\ell}{p}\, s!\, \frac{\partial^\ell J^{p-s}g_s }{\partial x^{i_1}\dots\partial x^{i_p}\partial\xi^{i_{p+1}}\dots\partial\xi^{i_\ell}}}_{\Jc_s}.
\end{align*}
First, let us focus on $\Jc_s$;
\begin{align*}
\Jc_s &= \sum_{p=s}^{\ell-1}(-1)^{p-s}\binom{p}{s}\binom{\ell}{p}\, s!\, \frac{\partial^\ell J^{p-s}g_s }{\partial x^{i_1}\dots\partial x^{i_p}\partial\xi^{i_{p+1}}\dots\partial\xi^{i_\ell}}\\
&=\sum_{p=0}^{\ell-s-1} (-1)^{p} \binom{p+s}{s}\, s! \binom{\ell}{p+s}\frac{\partial^\ell J^{p}g_s}{\partial x^{i_1}\dots\partial x^{i_{p+s}}\partial\xi^{i_{p+s+1}}\dots\partial\xi^{i_\ell}}\nonumber\\
		& = \sum_{p=0}^{\ell-s-1} (-1)^p \frac{\ell!}{p!(\ell-p-s)!}\,\frac{\partial^\ell J^{p}g_s}{\partial x^{i_1}\dots\partial x^{i_{p+s}}\partial\xi^{i_{p+s+1}}\dots\partial\xi^{i_\ell}}\nonumber\\
		&= \frac{\ell!}{(\ell-s)!}\sum_{p=0}^{\ell-s-1} (-1)^p \frac{(\ell-s)!}{p!(\ell-p-s)!}\,\frac{\partial^\ell J^{p}g_s}{\partial x^{i_1}\dots\partial x^{i_{p+s}}\partial\xi^{i_{p+s+1}}\dots\partial\xi^{i_\ell}}\nonumber\\
		&=\frac{\ell!}{(\ell-s)!} \sum_{p=0}^{\ell-s-1} (-1)^p \binom{\ell-s}{p} \,\frac{\partial^\ell J^{p}g_s}{\partial x^{i_1}\dots\partial x^{i_{p+s}}\partial\xi^{i_{p+s+1}}\dots\partial\xi^{i_\ell}}.
\end{align*} 
Next, recall the Lemma \ref{inversion} with values $ m=m-s $ and $r=\ell-s $ in equation \eqref{inversion_general} gives 
		\begin{align*}
		(J^0h_{m-\ell})_{i_1\dots i_{\ell-s}}= \frac{ (m-\ell)!}{(m-s)!}\sigma(i_1\dots i_{\ell-s})\sum_{p=0}^{\ell-s}(-1)^p\binom{\ell-s}{p}\, \frac{\partial^{\ell-s} J^ph}{\partial x^{i_1}\dots\partial x^{i_p}\partial\xi^{i_{p+1}}\dots\partial\xi^{i_{\ell-s}}}.
		\end{align*}
		Differentiating this relation $s$ times with respect to $ x^{i_{\ell-s+1}},\dots,x^{i_{\ell}}$ respectively and by applying the symmetrization  $\sigma(i_1\dots i_{\ell}) $, we obtain
		\begin{align*}\label{Eq 8.26}
		\sigma(i_1\dots i_{\ell}) \frac{\PD^{s} (J^0h_{m-\ell})_{i_1\dots i_{\ell-s}}}{\PD x^{i_{\ell-s+1}}\cdots \PD x^{i_{\ell}}}
		= \frac{ (m-\ell)!}{(m-s)!}\sigma(i_1\dots i_{\ell})\sum_{p=0}^{\ell-s}(-1)^p\binom{\ell-s}{p}\, \frac{\partial^{\ell} J^p h}{\partial x^{i_1}\dots\partial x^{i_{p+s}}\partial\xi^{i_{p+s+1}}\dots\partial\xi^{i_{\ell}}}.
		\end{align*}
This identity for $h = g_s$ reduces the expression for $\Jc_s$ to 
\begin{align*}
	\frac{(m-\ell)!}{m!}\sigma(i_1\dots i_\ell)	\Jc_s=  \frac{(-1)^{\ell-s+1}}{\binom{m}{\ell}(\ell-s)!}\frac{\partial^{\ell} J^{p}g_s}{\partial x^{i_1}\dots\partial x^{i_{\ell}}}+\frac{\binom{m-s}{\ell-s}}{\binom{m}{\ell}}\sigma(i_1\dots i_{\ell}) \frac{\PD^{s} ((J^0g_s)_{m-\ell})_{i_1\dots i_{\ell-s}}}{\PD x^{i_{\ell-s+1}}\cdots \PD x^{i_{\ell}}}.
\end{align*}
Now from Lemma \ref{John's condition }, we have 
\begin{align*}
			\Psi_{i_1\dots i_\ell}&= \frac{ (m-\ell)!}{m!}\sigma(i_1\dots i_\ell) \Bigg(\sum\limits_{p=0}^{\ell}(-1)^p\binom{\ell}{p}\,\frac{\partial^\ell \psi^p}{\partial x^{i_1}\dots\partial x^{i_p}\partial\xi^{i_{p+1}}\dots\partial\xi^{i_\ell}}\Bigg)\\
			&=\frac{ (m-\ell)!}{m!}\sigma(i_1\dots i_\ell) \sum\limits_{s=0}^{\ell-1}\Jc_s  +\frac{ (m-\ell)!}{m!}(-1)^\ell \frac{\partial^\ell \psi^p}{\partial x^{i_1}\dots\partial x^{i_\ell}}\\
&=\frac{(-1)^\ell}{\binom{m}{\ell}\ell!}\sum_{s=0}^{\ell-1} \frac{(-1)^{s+1}\ell!}{(\ell-s)!}\frac{\partial^{\ell} J^{p}g_s}{\partial x^{i_1}\dots\partial x^{i_{\ell}}}+\sigma(i_1\dots i_{\ell})\sum_{s=0}^{\ell-1}\frac{\binom{m-s}{\ell-s}}{\binom{m}{\ell}} \frac{\PD^{s} ((J^0g_s)_{m-\ell})_{i_1\dots i_{\ell-s}}}{\PD x^{i_{\ell-s+1}}\cdots \PD x^{i_{\ell}}}+ \frac{(-1)^\ell}{\binom{m}{\ell}\ell!} \, \frac{\partial^\ell \psi^{\ell}}{\partial x^{i_1}\dots\partial x^{i_{\ell}}}\nonumber\\
	&=\frac{(-1)^\ell}{\binom{m}{\ell}\ell!}\frac{\partial^{\ell} }{\partial x^{i_1}\dots\partial x^{i_{\ell}}}\left(\psi^{\ell}-\sum_{s=0}^{\ell-1} (-1)^{s}\binom{\ell}{s}\,s!\,J^pg_s\right)+\sigma(i_1\dots i_{\ell})\sum_{s=0}^{\ell-1}\frac{\binom{m-s}{\ell-s}}{\binom{m}{\ell}} \frac{\PD^{s} ((J^0g_s)_{m-\ell})_{i_1\dots i_{\ell-s}}}{\PD x^{i_{\ell-s+1}}\cdots \PD x^{i_{\ell}}}\\
	&=\frac{1}{\binom{m}{\ell}}\frac{\PD^{\ell}}{\PD x^{i_1}\cdots x^{i_{\ell}}} \chi^{\ell} + \frac{1}{\binom{m}{\ell}} \sigma(i_1\dots i_{\ell}) \sum_{s=0}^{\ell-1} \binom{m-s}{\ell-s}\, \frac{\PD^{s} ((J^0g_s)_{m-\ell})_{i_1\dots i_{\ell-s}}}{\PD x^{i_{\ell-s+1}}\cdots \PD x^{i_{\ell}}}.			
\end{align*}
This completes the proof of Lemma \ref{Lm 8.2}.
\end{proof}	
\begin{lemma}\label{restriction on Schwartz space}
	If $\psi^r$ is given by \eqref{eq:relation between psi and gs} for $ 0\le r \le \ell -1 $,	then for every indices $ 1\le i_1,j_1,\cdots, i_{m-\ell+1}, j_{m-\ell+1}\le n$, 
	\[ \left(	J_{i_1j_1} \cdots J_{i_{m-\ell+1}j_{m-\ell+1}} \chi^{\ell} \right)\Bigg|_{\tn}\in \mathcal{S}(\tn). \]
	\end{lemma}
\begin{proof}
	This proof follows from repeated application of \cite[Statement 2.10]{Krishnan2019a}.  We know from equation \eqref{homogenety w r to xi} of Lemma \ref{Lm5.1} that the function $ \chi^{\ell}(x,\xi) $ is positively homogeneous of degree $ \lambda=m-\ell-1 $ in its second variable. Also, by definition 
	\[ \chi^{\ell}\Big|_{\tn}\in \mathcal{S}(\tn).\]
	Differentiating \eqref{translation of chi l} with respect to $ t $, we obtain
	\begin{align*}
	\frac{d}{d t} \chi^{\ell}(x+t\xi,\xi)& = \l \xi,\PD_x\r \chi^{\ell}(x+t\xi,\xi) =0\\
	\Longrightarrow\qquad \qquad  \left.\frac{\d}{\d\, t} \chi^{\ell}(x+t\xi,\xi) \right|_{t =0}&= \l \xi,\PD_x\r \chi^{\ell}(x,\xi) =0.
	\end{align*}
	
	This implies $ \l \xi,\PD_x\r \chi^{\ell}(x,\xi) $ and all its derivatives with respect to $ x_j $'s and $ \xi_j $'s restricted to $ \tn $ belong to $ \mathcal{S}(\tn) $. Thus $ \chi^{\ell} $ satisfies the hypotheses of Lemma \ref{Lm 4.2} for $ k=p=m-\ell-1 $ and we get
	\begin{equation}\label{Eq5.27}
	\left(	J_{i_1j_1} \cdots J_{i_{m-\ell+1}j_{m-\ell+1}} \chi^{\ell} \right)\Bigg|_{\tn}\in \mathcal{S}(\tn).
	\end{equation}
	This finishes the proof.
\end{proof}

\subsection{Proof of Theorem \ref{th:range characterisation for I-k}}
\textbf{Proof of necessity.} 
To prove the necessary part of the theorem, let us assume that $(\vf^0, \vf^1, \dots, \vf^k) \in (\Sc(\tn))^{k+1}$ is in the range of operator $\Ic^k$, that is, there exists $ f \in \mathcal{S}(S^m)$ such that 
		\begin{align}\label{eq: relation between phil and Il}
	I^\ell f(x, \xi)=	\vf^{\ell}(x,\xi)= \int\limits_{-\infty}^\infty t^{\ell}\langle f(x+t\xi),\xi^m\rangle \,dt, \qquad \mbox{ for } 0 \leq \ell \leq k.
		\end{align}
Then the first condition \textit{(1)} of Theorem \ref{th:range characterisation for I-k}, $ \vf^{\ell}(x,-\xi) = (-1)^{m-\ell} \vf^{\ell}(x,\xi)$ for $ \ell=0,1,\dots,k$,    can be verified by a straight forward substitution ($\xi \mapsto -\xi$). And the second condition $\textit{(2)}$ of Theorem \ref{th:range characterisation for I-k} follows from \cite[Lemma 2.5]{Krishnan2019a}, which can be proved by a direct computation too. \vspace{2.5 mm}\\
\noindent \textbf{Proof of sufficiency.}  Assume $(\vf^0, \vf^1, \dots, \vf^k) \in (\Sc(\tn))^{k+1}$  satisfy the properties \textit{(1)} and \textit{(2)}, then our aim is to find a $f  \in \Sc(S^m)$ such that equation \eqref{eq: relation between phil and Il} holds. Our idea is to construct a tensor field $f$ of the form 
		\begin{equation}\label{Eq5.1}
		f= \sum_{s=0}^{k}\d^{s} g_{s}, \quad \mbox{ where } g_s \in \Sc(S^{m-s}), \mbox{ for } s = 0, 1, \dots , k
		\end{equation} 
which will be equivalent to find tensor fields $g_s$, for $0 \leq s \leq k$, satisfying the following relation (see the discussion in the first two paragraphs of subsection \ref{subsec:lemmas for range characterization}): 
		\begin{equation}\label{Eq5.2}
		\begin{aligned}
		\vf^\ell=I^\ell\!f &= \sum_{s=0}^{\ell} (-1)^s\,
		\binom{\ell}{s}\, s! \,I^{\ell-s} g_s,  \qquad \mbox{ for }  0\le \ell\le k.
			\end{aligned}
		\end{equation}
To obtain required tensor fields $g_s$, we are going to use Mathematical induction and the Theorem \ref{Th3.1} successively. In particular, we show that the function $ \chi^{\ell} \in C^{\infty}(\Rn \times\Rn\setminus{0})$ (for each fixed $0\leq  \ell \le k$)  defined  in Lemma \ref{Lm5.1} by
	  \begin{equation}\label{definition_of_chi_l}
	  \chi^{\ell}=\frac{(-1)^\ell}{\ell!} \left(\psi^\ell - \sum_{s=0}^{\ell-1}(-1)^{s}\,\binom{\ell}{s}\,s! \,J^{\ell-s} g_{s}\right)
	  \end{equation}
	  satisfies the hypotheses of the Theorem \ref{Th3.1} when $m$ replaced by $ m-\ell$. Then by applying Theorem \ref{Th3.1} on $\chi^{\ell}$, we will prove the existence of tensor fields $g_s$ ($0\leq s \le k$) iteratively. \vspace{2 mm}\\ 
\textbf{Case $\ell =0$:} For $ \ell=0 $, we have  $ \chi^0=\psi^0 $. Also, the statement of theorem (equation \eqref{Johns's condition for psi k}) gives
			\[  J_{i_{m+1}j_{m+1}}\cdots J_{i_2j_2} J_{i_1j_1} \psi^k=0.  \]
Using \eqref{Eq4.1} for $\ell=k$ together with the fact that the operators $ \l \xi, \PD_{x}\r $ and $ J_{ij} $ commute, we get
			\[J_{i_{m+1}j_{m+1}}\cdots J_{i_2j_2} J_{i_1j_1} \psi^0=0. \]
			Hence, we can apply Theorem \ref{Th3.1} on $ \chi^0 $ to get $ g_0\in \mathcal{S}(S^m) $ satisfying \[ I^0g_0 =\vf^0 \quad \mbox{and}\quad J^0g_0=\psi^0=\chi^0. \] 
			Thus, we have proved that our claim about the function $\chi^\ell$ is valid for $ \ell=0 $.\vspace{2 mm}\\ 
\textbf{Induction hypothesis:} Assume that the function $ \chi^{r}$ satisfies the properties \textit{(1)} and \textit{(2)} of Theorem \ref{Th3.1} and hence shows the existence of $g_{r}\in \mathcal{S}(S^{m-r})$ such that 
	 \begin{equation*}
	 \chi^{r} = J^0g_r= \frac{(-1)^r}{r!} \left(\psi^r - \sum_{s=0}^{r-1}(-1)^{s}\binom{r}{s}s!\, J^{r-s} g_{s}\right), \qquad \mbox{ for } 0\leq r \leq \ell -1. 
	 \end{equation*}
	 This implies
	 \begin{equation}\label{Eq 8.13}
	 \begin{aligned}
	 \psi^r=&\sum_{s=0}^{r-1}(-1)^{s}\binom{r}{s}\,s!\, J^{r-s} g_{s}+ (-1)^r r! J^0g_r
	 =&\sum_{s=0}^{r}(-1)^{s}\binom{r}{s}s!\,J^{r-s} g_{s}.
	  \end{aligned}
	 \end{equation}
Here, we use Lemma \ref{John's condition } to get the following relation \[J_{i_1j_1} \cdots J_{i_{m-\ell+1}j_{m-\ell+1}} \Psi_{i_1\cdots i_{\ell}} =0.\]
		This together with Lemma \ref{Lm 8.2} entails
		\begin{align}\label{Eq8.5}
		J_{i_1j_1} \cdots J_{i_{m-\ell+1}j_{m-\ell+1}}\left(  \frac{\PD^{\ell}}{\PD x^{i_1}\cdots x^{i_{\ell}}} \chi^{\ell} +  \sigma(i_1\dots i_{\ell}) \sum_{s=0}^{\ell-1}\binom{m-s}{\ell-s} \frac{\PD^{s} ((J^0g_s)_{m-\ell})_{i_1\dots i_{\ell-s}}}{\PD x^{i_{\ell-s+1}}\cdots \PD x^{i_{\ell}}}\right)=0.
		\end{align}
The operator $J_{ij}$ is a constant coefficients differential operator and commutes with partial derivatives (in fact with any constant coefficient differential operator). Therefore \eqref{Eq8.5} is equivalent to the equation
		\begin{equation}\label{Eq8.6}
\begin{aligned}
	&	\frac{\PD^{\ell}}{\PD x^{i_1}\cdots x^{i_{\ell}}} \left(	J_{i_1j_1} \cdots J_{i_{m-\ell+1}j_{m-\ell+1}} \chi^{\ell}\right) \\&\qquad \quad +  \sigma(i_1\dots i_{\ell}) \sum_{s=0}^{\ell-1} \binom{m-s}{\ell-s} \frac{\PD^{s} }{\PD x^{i_{\ell-s+1}}\cdots \PD x^{i_{\ell}}} \left( J_{i_1j_1} \cdots J_{i_{m-\ell+1}j_{m-\ell+1}} ((J^0g_s)_{m-\ell})_{i_1\dots i_{\ell-s}}\right)=0.
		\end{aligned}
		\end{equation} 
The second term on the left side of above equation is zero from the Proposition \ref{general johns condition}. Thus the above relation reduces to
		\begin{equation}\label{Eq 8.14}
		\frac{\PD^{\ell}}{\PD x^{i_1}\cdots x^{i_{\ell}}}\left(J_{i_1j_1} \cdots J_{i_{m-\ell+1}j_{m-\ell+1}}  \chi^{\ell} \right)=0.
		\end{equation}
For a fixed $ \xi \ne 0 $, the Proposition \ref{restriction on Schwartz space} implies  that  the restriction of $ 	\left(J_{i_1j_1} \cdots J_{i_{m-\ell+1}j_{m-\ell+1}}  \chi^{\ell} \right)(\cdot,\xi) $ on the hyperplane  $ \xi^{\perp}=\{x\in \Rn: \l x,\xi \r=0 \}  $ belongs to $ \mathcal{S}(\xi^{\perp}) $. Moreover,  \eqref{Eq 8.14} implies that the restriction is itself zero. Since $ \xi\ne 0 $ is arbitrary, therefore		\begin{equation*}
		J_{i_1j_1} \cdots J_{i_{m-\ell+1}j_{m-\ell+1}}\chi^{\ell}=0.
		\end{equation*}
		Lemma \ref{extension lemma} gives the function $ \chi^{\ell} $ obtained from $ \tilde{\chi}^{\ell} =\chi^{\ell}|_{\tn}$ by the following rule 
		\[ \chi^{\ell}=|\xi|^{m-\ell-1}\,\widetilde{\chi}^{\ell}\left( x-\frac{\l x,\xi \r }{|\xi|^2}\xi, \frac{\xi}{|\xi|}   \right), \]
		where $\tilde{\chi}^{\ell}= \frac{(-1)^l}{l!} \left(\vf^l - \sum_{s=0}^{\ell-1}(-1)^{s}\binom{\ell}{s}\,s!\, I^{\ell-s} g_{s}\right) \in \mathcal{S}(\tn)$.  And	equation \eqref{homogenety w r to xi} gives 
	\[\chi^{\ell}(x,-\xi)= (-1)^{m-\ell} \chi^{\ell}(x,\xi). \]
	This implies $ \tilde{\chi}^{\ell}=\chi^{\ell}|_{\tn} $ also enjoys the same. Thus we have prove that if \eqref{Eq 8.13} holds, then the function $ \chi^{\ell} $ satisfies the hypotheses of Theorem \ref{Th3.1} for $ m-\ell $ tensor fields. With the help of second principle of induction, there exists a $ g_{\ell}\in \mathcal{S}(S^{m-\ell}) $ such that
	\begin{equation}\label{identity for chi l}
	J^0g_{\ell}=\chi^{\ell} \quad I^0g_{\ell}=\chi^{\ell}|_{\tn}
	\end{equation}
holds 	for $ \ell =0,\dots,k $. Thus, the tensor field $ f $ defined by \eqref{Eq5.1} gives 
		\[ I^{\ell}f=\vf^{\ell};\quad \ell=0,1,\dots,k. \]
		This completes the proof of sufficient part and Theorem \ref{th:range characterisation for I-k} as well.
\begin{remark}
\begin{itemize}
\item We already know from Theorem \ref{th:decomp_in_original_sp} that the operator $\Ic^k$ has an infinite dimensional kernel containing all $(k+1)$-potential field (tensor fields of the form $ f=\d^{k+1}v $ for some $ v\in\mathcal{S}(S^{m-k-1})$). This was the primary motivation to define $ f $ in the form of \eqref{Eq5.1} for the proof of Theorem \ref{th:range characterisation for I-k}.
\item The techniques used in this work to prove the range characterization for the operator $\Ic^k$  is different from the one used in \cite[Theorem 1.3]{Krishnan2019a}. In fact for the  particular case $k =m$, our result provides a new proof of \cite[Theorem 1.3]{Krishnan2019a}. Hence, the Theorem \ref{th:range characterisation for I-k} can be thought of as a  generalization of \cite[Theorem 1.3]{Krishnan2019a}.
\end{itemize}
\end{remark}

\bibliographystyle{plain}
\bibliography{reference}
	
\end{document}